\documentclass[12pt,a4paper]{amsart}
\usepackage[utf8]{inputenc}
\usepackage[a4paper,inner=2.8cm,outer=2.5cm,top=2.8cm,bottom=2.5cm]{geometry}                
\usepackage{amsmath,amssymb,amscd,amsthm,amsfonts,color,enumerate,fix-cm,layout,mathrsfs,mdwlist,stmaryrd,tikz,xspace,comment}
\usepackage{graphicx}
\usepackage{hyperref}
\usepackage[all]{xy}
\usepackage[bottom]{footmisc}
\usepackage{ifthen}
\usepackage{xfrac}
\theoremstyle{plain}                 
\newtheorem{theorem}{Theorem}[section]     
\newtheorem{proposition}[theorem]{Proposition} 
\newtheorem{corollary}[theorem]{Corollary}     
\newtheorem{lemma}[theorem]{Lemma}        
\theoremstyle{definition}           
\newtheorem{definition}[theorem]{Definition}    
\newtheorem{example}[theorem]{Example} 
\theoremstyle{remark}       
\newtheorem{remark}[theorem]{Remark}    
\hypersetup{colorlinks=true,linkcolor=blue,citecolor=purple,filecolor=magenta,urlcolor=cyan}

\def\bea{\begin{eqnarray}}
\def\eea{\end{eqnarray}}
\def\nn{\nonumber}

\begin{document}

\title[Electrical networks, Lagrangian Grassmannians and symplectic groups]{Electrical networks, Lagrangian Grassmannians and symplectic groups}

\author[B.~Bychkov]{B.~Bychkov}
\address{B.~B.: current affiliation: Department of Mathematics, University of Haifa, Mount
Carmel, 3488838, Haifa, Israel. Faculty of Mathematics, National Research University Higher School of Economics, Usacheva 6, 119048 Moscow, Russia; and Centre of Integrable Systems, P.~G.~Demidov Yaroslavl State University, Sovetskaya 14, 150003, Yaroslavl, Russia}
\email{bbychkov@hse.ru}
\author[V.~Gorbounov]{V.~Gorbounov}
\address{V.~G.: Faculty of Mathematics, National Research University Higher School of Economics, Usacheva 6, 119048 Moscow, Russia}
\email{vgorb10@gmail.com}
\author[A.~Kazakov]{A.~Kazakov}
\address{A.~K.: Lomonosov Moscow State University, Moscow, Russia; Moscow Institute of Physics and Technology (National Research University), Russia; and Centre of Integrable Systems, P.~G.~Demidov Yaroslavl State University, Sovetskaya 14, 150003, Yaroslavl, Russia }
\email{anton.kazakov.4@mail.ru}
\author[D.~Talalaev]{D.~Talalaev}
\address{D.~T.: Lomonosov Moscow State University, Moscow, Russia; and Centre of Integrable Systems, P.~G.~Demidov Yaroslavl State University, Sovetskaya 14, 150003, Yaroslavl, Russia}
\email{dtalalaev@yandex.ru}

\begin{abstract}
 We refine the result of T. Lam \cite{L} on embedding the space $E_n$ of electrical networks on a planar graph with $n$ boundary points into the totally non-negative Grassmannian $\mathrm{Gr}_{\geq 0}(n-1,2n)$ by proving first that the image lands in $\mathrm{Gr}(n-1,V)\subset \mathrm{Gr}(n-1,2n)$ where $V\subset \mathbb{R}^{2n}$ is a certain subspace of dimension $2n-2$. The role of this reduction in the dimension of the ambient space is crucial for us. We show next that the image lands in fact inside the Lagrangian Grassmannian $\mathrm{LG}(n-1,V)\subset \mathrm{Gr}(n-1,V)$. As it is well known $\mathrm{LG}(n-1)$ can be identified with $\mathrm{Gr}(n-1,2n-2)\cap \mathbb{P} L$ where $L\subset \bigwedge^{n-1}\mathbb R^{2n-2}$ is a subspace  of dimension equal to the Catalan number $C_n$, moreover it is the space of the fundamental representation of the symplectic group $\mathrm{Sp}(2n-2)$ which corresponds to the last vertex of the Dynkin diagram. We show further that the linear relations cutting the image of $E_n$ out of $\mathrm{Gr}(n-1,2n)$  found in \cite{L} define that space $L$. This connects the combinatorial description of $E_n$ discovered in \cite{L} and representation theory of the symplectic group.
 
\end{abstract}

\maketitle

MSC2020: 14M15, 82B20, 05E10

Key words: Electrical networks, Electrical Algebra, Lagrangian Grassmanian.
\tableofcontents

\hspace{0.1in}

\section{Introduction}
The connection between the theory of electrical networks and the Lie theory was discovered in the pioneering work \cite{LP}. Later in \cite{L} an embedding $X$ of the space of electrical networks $E_n$ on a planar circular graph with $n$ boundary points to the non-negative part of the Grassmannian $\mathrm{Gr}(n-1,2n)$ was constructed using the technique developed in the work of A. Postnikov. Furthermore its image $X(E_n)$ was explicitly identified in combinatorial terms, and its closure was studied. 

Of particular importance for us there will be the operations of attaching an edge and a spike to the graph of a network. The change of $E_n$ under these operations was described in \cite{L} as an action of a certain group, which we will refer to as the Lam group $\mathcal El_n$, on the space of dimension $2n$ and therefore on $\mathrm{Gr}(n-1,2n)$. This group is closely related to the electrical group introduced in \cite{LP}.

Another description of $E_n$ was obtained in \cite{GT} using the technique from the theory of integrable models. This description provided a different embedding of $E_n$ into the same Grassmannian. 

In this paper we will construct a new parametrisation of $X(E_n)$  using the ideas from the work of Lam \cite{L} and Postnikov \cite{P}. First we will show that $X(E_n)$  naturally sits in a smaller manifold 
$\mathrm{Gr}(n-1,V)\subset \mathrm{Gr}(n-1,2n)$ where $V\subset \mathbb R^{2n}$ is a certain subspace of dimension $2n-2$. The role of this reduction in the dimension of the ambient space is crucial for us because next we show that $X(E_n)$
sits in the Lagrangian Grassmannian $\mathrm{LG}(n-1,V)\subset \mathrm{Gr}(n-1,V)$. Since the dimension of $\mathrm{LG}(n-1,V)$
is equal to the dimension of $E_n$ the former 
is a natural home for the latter. We prove that the image of the map $X$ lands in $\mathrm{LG}_{\geq 0}(n-1,V)$, but the question whether it is a bijection remains unanswered. We are planning to return to it in a separate paper.

As it is known from geometric representation theory $\mathrm{LG}(n-1)$ can be identified with $\mathrm{Gr}(n-1,2n-2)\cap \mathbb{P} L$ where $L\subset \bigwedge^{n-1}\, \mathbb R^{2n-2}$ is the space of the fundamental representation of the group $\mathrm{Sp}(2n-2)$, which corresponds to the last vertex of the Dynkin diagram.
Its dimension is equal to the Catalan number $C_n$.
We will interpret in these terms the description of $X(E_n)$ as the slice of the Grassmannian $\mathrm{Gr}(n-1,2n)$ by a projectivisation  of the subspace $H$ of dimension equal to the Catalan number $C_n$ obtained in \cite{L}. It turns out 
that $V$ is invariant under the action of the group $\mathcal El
_n$, which factors through the action of the symplectic group $\mathrm{Sp}(2n-2)$.
Furthermore $H$
is a subspace of $\bigwedge^{n-1} V$ and it  is invariant under the action of $\mathrm{Sp}(2n-2)$ as well.
Therefore its projectivisation becomes the space for the Plucker embedding of $\mathrm{LG}(n-1)$ as the orbit of the highest weight vector in the fundamental representation of the group $\mathrm{Sp}(2n-2)$ which corresponds to the last vertex of the Dynkin diagram.
It means in particular that the linear relations cutting out $X(E_n)$ from $\mathrm{Gr}(n-1,2n)$  which were found in \cite{L} define
the space of the above fundamental representation of $\mathrm{Sp}(2n-2)$ connecting therefore the combinatorics discovered in \cite{L} and representation theory of the symplectic group.
\footnote{After the work described in this paper was finished Thomas Lam informed us in a private communication that he obtained the result that $X(E_n)\subset \mathrm{LG}(n-1,V)$ when \cite{L} was already published. 
We have learned as well a similar result was obtained in a recent preprint \cite{CGS}.
However the geometric interpretation of the space $H$ and the connection to the representation theory of the symplectic group appear to be new. 
}

The description of $X(E_n)$ given in \cite{GT} also can be  naturally identified with the subset of $\mathrm{LG}(n-1)$. Finally we show that the descriptions of the $X(E_n)$ given in \cite{GT} and \cite{L} are related by an isomorphism of the ambient space $\mathbb R^{2n-2}$.

As it was pointed out in \cite{GP} there is a striking similarity of the parametrisation  of the image of the space $E_n$ in $\mathrm{Gr}(n-1,2n)$ and the parametrisation of the space of the Ising models on a planar circular graph embedded in the orthogonal Grassmannian $\mathrm{OG}(n,2n)$ constructed in \cite{GP}. Our results may provide some insight into it. 
We are planning to explore it in future publications.

\subsection{Organisation of the paper}

In Section \ref{background} we have collected the  basic facts from the theory of electrical networks and some important for us constructions due to T. Lam and A. Postnikov.

In Section \ref{sec:networks}  we introduce a new parametrisation of the space $E_n$ in the Grassmannian $\mathrm{Gr}(n-1,2n)$. 
We study the compactification of $X(E_n)$ introduced in \cite{L} in terms of our parametrisation.

In Section \ref{sec:Lagr} we prove that $X(E_n)$ naturally sits inside $\mathrm{LG}(n-1,V)$,  where $V\subset \mathbb R^{2n}$ is a subspace of co-dimension $2$.

In Section \ref{sec:representation} we obtain the main result of the paper. We show that the linear space $H$ from \cite{L} becomes the space of the projective embedding of $\mathrm{LG}(n-1, V)$, moreover $H$ is invariant under the action of the electrical group $\mathcal{E}l_n$ which factors through the action of the symplectic group $\mathrm{Sp}(2n-2)$, and, as a representation of $\mathrm{Sp}(2n-2)$, it is the fundamental representation which corresponds to the last vertex of the Dynkin diagram.

In Section \ref{Sec:nonneg} we show that the image of $X(E_n)$ lands in $\mathrm{Gr}_{\geq 0}(n-1,V)$, the non-negative part of $\mathrm{Gr}(n-1,V)$. This is a non-trivial refinement of the fact $X(E_n)\subset \mathrm{Gr}_{\geq 0}(n-1,2n)$ proved in \cite{L}. On the other hand $X(E_n)$ does not land in $\mathrm{LG}_{\geq 0}(n-1)$ studied in \cite{K} as we will explain in this section.

In Section \ref{Sec:app} we give some useful applications of our parametrisation of $X(E_n)$. Namely we will prove the well-known  formulas from \cite{CIM} for the change of the response matrix entries after adjoining a boundary spike or a boundary bridge and give a simple proof of Theorem $1.1$ from \cite{KW},  which expresses Laurent polynomial phenomenon for the special ratios of groves partition functions.

In Section \ref{Sec:vertex} following the work \cite{GT} we describe the natural connection between the boundary measurements matrices assigned to elements of $E_n$ and points in $\mathrm{LG}(n-1,V)$. We show that the embedding from \cite{GT} and the Lam embedding can be identified. And finally we present another proof of Theorem \ref{laggr} and Theorem \ref{nonneglagr} using  standard well-connected graph technique.

\section{Electircal networks}
\label{background}

\subsection{Postnikov's approach and Lam's embedding}

We start with some background from \cite{L}, \cite{LT} and \cite{P}.

\begin{definition}  \label{l1}
A planar circular graph $\Gamma$ together with a weight function $\omega:E(\Gamma)\rightarrow \mathbb R_{\geq}$ is called a {\em bipartite network} $N(\Gamma, \omega)$ if the following holds:

\begin{itemize}
    \item The boundary vertices $V_0(\Gamma)$ are labeled by positive integers clockwise;
    \item The degrees of the boundary vertices are all equal to one and the weights of the edges incident to the boundary vertices are all equal to one;
    \item All the vertices $v\in V(\Gamma)$ are coloured by either black or white colour and each edge is incident to the vertices of different colours.
\end{itemize}
\end{definition}

\begin{definition} An {\em almost perfect matching} $\Pi$ is a collection of edges of $\Gamma$ such that
\begin{itemize}
\item Each interior vertex is incident to one of the edges in $\Pi$;
\item boundary vertices may or may not be incident to the edges in $\Pi$.
\end{itemize}
The weight $\mathrm{wt}(\Pi)$ of a matching $\Pi$ is a product of the weights of all edges in the matching. 
\end{definition}
For $I\subset V_0(\Gamma)$ denote by $\Pi(I)$ the set of almost perfect matchings  such that the black vertices in $I$ are {\it incident} to the edges in $\Pi$ and the white vertices in $I$ are {\it not incident} to the edges from $\Pi$.

Define the following invariant of a bipartite network $N(\Gamma, \omega)$:

$$k(\Gamma)=\frac{1}{2}\left(n+\sum\limits_{v \in black}(\mathrm{deg}(v)-2)+\sum\limits_{v \in white}(2-\mathrm{deg}(v))\right),$$
here $n$ is a number of boundary vertices of the graph $\Gamma.$

For a bipartite network  $N(\Gamma, \omega)$ define a collection of boundary measurements as follows:
\begin{definition} \label{ld}
For each $I\subset V_0(\Gamma)$ such that $|I|=k(G) $ the {\em boundary measurement} defined by $I$ is given by the formula:
\begin{equation}
\Delta_{I}^M=\sum \limits_{\Pi\in \Pi(I)}\mathrm{wt}(\Pi).
\end{equation}
The superscript $M$ means that it is a boundary measurement of an almost perfect matching.
\end{definition}

\begin{definition}
The \textit{totally non-negative Grassmanian} (we usually call it just non-negative Grassmanian) is the locus in the Grassmanian with the non-negative Plucker coordinates.
\end{definition}
 
\begin{theorem}\cite{LT}\label{l2}
For a bipartite network $N(\Gamma, \omega)$ the collection of boundary measurements $\Delta_{I}^M$ considered as a set of Plucker coordinates defines a point in the non-negative Grassmannian $\mathrm{Gr}(k(\Gamma), n)$. In particular the boundary measurements $\Delta_{I}^M$ satisfy the Plucker relations.
\end{theorem}

We will need a different way to calculate boundary measurements $\Delta_{I}^M$. For this purpose we define the notion of a flow. 

\begin{definition}
Let $N(\Gamma, \omega)$ be a bipartite network  with an orientation $O$ on the set of edges of $\Gamma$ which satisfies the following:
\begin{itemize}
    \item  Each internal black vertex has exactly one outgoing edge;
    \item  Each internal white vertex has exactly one incoming edge.
\end{itemize}
Such an orientation we call {\em perfect} and the network a {\em perfect network} and denote by $N(\Gamma, \omega, O)$.
\end{definition}

\begin{remark} \label{bij}
For each bipartite graph there is a natural bijection between the set of perfect orientations and the set of almost perfect matchings. Consider an almost perfect matching $\Pi$ on a bipartite graph $\Gamma$. We construct a perfect orientation $O(\Pi)$ of $\Gamma$ as follows:
 \begin{itemize}
     \item if $e \notin \Pi,$ then we orient the edge $e$ from white to black;
     \item if $e \in \Pi,$ then we orient the edge $e$ from black to white.
 \end{itemize}  
 It is easy to see that the constructed map is a bijection.
\end{remark}

\begin{definition} \label{pf}
A subset $F$ of the set of edges $E(\Gamma)$ of a perfect network $N(\Gamma,\omega,O)$ is called a {\em flow} if for each internal vertex the number of incoming edges is equal to the number of outgoing edges.
The weight of a flow $\mathrm{wt}(F)$ is the product of the weights of its edges.
\end{definition}
In the perfect network $N(\Gamma,\omega,O)$ the set of boundary vertices is naturally split into the sources and the sinks. 
The following statement  is obvious from the definition 
\begin{proposition} \label{pf1}
Each flow is a union of paths connecting the sources and the sinks and cycles. 
\end{proposition}

For a set $I\subset V_0$, denote by $F(I)$  the set of flows $F$ such that the sources which belong to the set $I$ are not incident to any edge of $F$ and the sinks which belong to the set $I$ are incident to some edge in $F$.

\begin{definition}\label{lp1}
The {\em boundary flow measurements} for a perfect network with chosen set $I\subset V_0$ of cardinality $k(\Gamma)$ is defined by the formula:
$$\Delta_{I}^F=\sum \limits_{F\in F(I)}\mathrm{wt}(F).$$
\end{definition}
We have the following 
\begin{theorem} \cite{LT}, \cite{P} \label{lpt}
For a perfect network $N(\Gamma,\omega,O)$ the collection of boundary flow measurements $\Delta_{I}^F$ considered as a set of Plucker coordinates defines a point in the non-negative Grassmannian $\mathrm{Gr}(k(\Gamma), n)_{\geq 0}$.  In particular the boundary measurements $\Delta_{I}^{F}$ satisfy the Plucker relations.
\end{theorem}
We have two embeddings of the space of weighted graphs which are equipped with a bipartite and a perfect network structure into $\mathrm{Gr}(k(\Gamma), n)_{\geq 0}$. The following theorem compares them:
\begin{theorem}\cite{LT}, \cite{P} \label{link}
Let $N(\Gamma,\omega)$ and $N(\Gamma,\omega',O)$ be a bipartite and a perfect network structure on the same graph and the weight functions $\omega$ and $\omega'$ are connected as follows:
    \begin{itemize}
        \item for any $e\in E(\Gamma)$ in the network  $N(\Gamma, \omega', O)$ connecting the white and the black vertices the weight $\omega(e)$ is the same as it is in the bipartite graph $N(\Gamma, \omega)$
        \item for any $e\in E(\Gamma)$ in the network $N(\Gamma, \omega', O)$ connecting the black and the white vertices the weight $\omega(e)$ is reciprocal of the weight of the same edge in the bipartite graph $N(\Gamma, \omega)$
    \end{itemize}
Then the Plucker coordinates defined by networks $N(\Gamma,\omega)$ and $N(\Gamma,\omega',O)$ relate with each other as follows:
$$\Delta_{I}^F=\frac{\Delta_{I}^M}{\mathrm{wt}(\Pi_0)},$$
where $wt(\Pi_0)$ is the weight of the almost perfect matching corresponding to the perfect orientation $O$, see Remark \ref{bij}.


In other words, these two networks define the same point of    $\mathrm{Gr}(k(\Gamma), n)_{\geq 0}.$
\end{theorem}

Postnikov also considered a larger class of networks defined below. We will call them the {\em Postnikov networks}.
\begin{definition} \label{matrpost}
Let $(\Gamma, \omega, O)$ be an oriented graph with the orientation (not necessary perfect)  and weight function $\omega$. We will call such set a {\em Postnikov network} and denote it by $PN(\Gamma,\omega, O)$ if 
\begin{itemize}
    \item $\omega:E(\Gamma)\rightarrow \mathbb R_{\geq}$;
    \item  There are $k$ sources and $n-k$ sinks on the boundary and the boundary vertices are labeled by positive integers clockwise;
    \item  The degrees of the boundary vertices are all equal to one and the weights of the edges incident to the boundary vertices are all equal to one.
\end{itemize}
\end{definition}
The labeling of the boundary vertices induces the labeling of the set of sources $I=\{i_1,\dots,i_k\}$.
\begin{definition}
  Consider a  Postnikov network  $PN(\Gamma,\omega, O)$ and denote by $M_{i_rj}$  the sum of weights of the directed paths from $i_r$ to $j$, which is defined as follows:
  $$M_{i_rj}=\sum_{p:i_r \to j}(-1)^{wind(p)}wt(p),$$
  here $wt(p)$ is the product of all edges belonging to $p$, and $wind(p)$ is the winding index of a path $p$, which is the signed number of full turns of a path $p,$ counting counterclockwise turns as positive, see \cite{P} for more details.   
\end{definition}
\begin{definition}
\label{def:extendedmatr}
The {\em extended matrix of boundary measurements} $A$ for a Postnikov network $PN(\Gamma,\omega, O)$ is a $k \times n$ is defined as follows:
 \begin{itemize}
     \item submatrix formed by the columns labeling the sources is the identity matrix;
     \item otherwise $a_{rj}=(-1)^sM_{i_rj},$ where $s$ is the number of elements from $I$ between the source labeled $i_r$ and the sink labeled $j$. 
 \end{itemize}
\end{definition} 

 \begin{theorem} \cite{P}
For a Postnikov network $PN(\Gamma,\omega, O)$ the minors $\Delta_I(A)$ of the extended boundary measurements matrix $A$ are non-negative, hence $A$ defines a point of the non-negative Grassmannian $\mathrm{Gr}(k(\Gamma), n)_{\geq 0}$.
 \end{theorem}
By construction the above point depends on the choice of the orientation, however if  we are working with the bipartite graphs and perfect orientations on them this dependence is not essential. Namely we have the following: 

\begin{theorem} \cite{P} \label{chanorient}
 Let $\Gamma $ be a bipartite graph and $O_1$ and $O_2$ be  two different perfect orientations on $\Gamma.$  Consider two Postnikov networks $PN_1(\Gamma, \omega_1, O_1),$  $PN_2(\Gamma, \omega_2, O_2)$ such that
 \begin{itemize}
     \item  All vertices $v\in V(\Gamma)$ are colored in black or white and each edge connects vertices of different colours;
     \item  Both orientations  $O_1$ and $O_2$ are perfect;
     \item If the orientations of $e \in \Gamma$ are the same in both networks then $w_1(e) = w_2(e)$ and $w_1(e)w_2(e)=1$ otherwise.
 \end{itemize}
 
Then the minors of the extended boundary measurement matrices $A_1$ and $A_2$ are related as follows:
 $$\Delta_J(A_2)=\frac{\Delta_J(A_1)}{\Delta_{I_2}(A_1)},$$
 where $I_2$ is the set of indices of the sources of the second network.
 \end{theorem}
 
\subsection{Electrical networks}

In this subsection we remind the reader of some basic fact of the theory of electrical networks.
\begin{definition} 
An electrical network $e(\Gamma, \omega)$ is a planar graph embedding into a disk, which satisfies the following conditions:  
\begin{itemize}
    \item all vertices are divided into the set of interior vertices and the set of exterior or boundary vertices and all boundary vertices lie on a circle;
    \item boundary vertices are enumerated counterclockwise;
    \item every edge of $\Gamma$ is equipped with a positive weight which is the conductivity of this edge .
\end{itemize}
We will write that $e \in E_n$ if $\Gamma$ has $n$ boundary vertices.

When electrical voltage is applied to the boundary vertices it extends to the interior vertices and the electrical current will run through the edges of $\Gamma.$ These voltages and electrical currents are defined by the Kirchhoff laws  \cite{CIM}.
\end{definition}
One of the major objects in the theory is the {\it response matrix}.
\begin{theorem}\cite{CIM}
Consider an electrical network and apply to its boundary vertices voltage  $\textbf{U}.$ This voltage induces  the current $\textbf{I}$ running through  the boundary vertices. There is a  matrix $M_R(e)$ which relates  $\textbf{U}$ and $\textbf{I}$ as follows:
$$M_R(e)\textbf{U}=\textbf{I}.$$

This matrix is called the response matrix. 
\end{theorem} 
We will need the following important properties of the response matrix.

\begin{theorem}\cite{CIM} \label{respc}
Let $e \in E_n $ and $M_R(e)$ be its response matrix. Then the following holds:
\begin{itemize}
    \item $M_R(e)$ is a symmetric matrix;
    \item non-diagonal entries are non-negative;
    \item sum of entries of each column and each row is equal to $0.$ 
\end{itemize}
\end{theorem}
Now we are ready to define an embedding of $E_n$
to $\mathrm{Gr}(k(\Gamma), n)_{\geq 0}$. 

\begin{definition} \label{def-grov1}
 A {\em grove} $G$ on $\Gamma$ is a spanning subforest that is an acyclic subgraph that uses all vertices such that each connected component $G_i \subset G$ contains at least one boundary vertex.
The {\em boundary partition} $\sigma(G)$ is the set partition of $\{\bar 1, \bar 2, . . . , \bar n\}$ which specifies which
boundary vertices lie in the same connected component of $ 
G$. Note that since $\Gamma$ is planar,
 $\sigma(G)$ must
be a {\em non-crossing partition}, also called a {\em  planar set partition}.
\end{definition}
We will often write set partitions in the form $\sigma = (\bar a\bar b\bar c \bar d \bar e| \bar f \bar g|\bar h)$. Denote by $\mathcal NC_n$ the set of non-crossing partitions on $\{\bar 1,\ldots, \bar n\}$.

Each non-crossing partition $\sigma$ on $\{\bar 1, \bar 2, . . . , \bar n\}$ has a {\em dual non-crossing partition} on $\{\tilde 1, \tilde 2, . . . , \tilde n\}$ where by
convention $\tilde i$ lies between $\bar i$ and $\overline {i + 1}$. For example $(\bar1,\bar 4, \bar 6|\bar 2, \bar 3|\bar 5)$ is dual to $(\tilde{1}, \tilde{3}|\tilde{2}|\tilde{4}, \tilde{5}|\tilde{6})$.
\begin{definition} \label{def-grov2}
For a non-crossing partition $\sigma$ we define the grove measurement related to $\sigma$ as follows
$$L_{\sigma(\Gamma)} := \sum_{G|\sigma(G)=\sigma} \mathrm{wt}(G),$$
where the summation is over all groves with boundary partition $\sigma$, and $\mathrm{wt}(G)$ is the product of weights of edges in $G$.
\end{definition}

 In particular we will need the groves measurements of the following non-crossing partitions whose definition is clear from the notations:
 \begin{itemize}
     \item $L:=L_{*|*|*\dots}$;
     \item $L_{ij}:=L_{ij|*|*\dots}$;
     \item  $L_{kk}:=\sum_{i\neq k}L_{ik}.$
 \end{itemize}

\begin{definition} \label{def-elb}
For $e\in E_n$ we define a planar bipartite network $N(\Gamma, \omega)$ embedded into the disk. 
If $\Gamma$ has boundary vertices $\{\bar 1, \bar 2, . . . , \bar n\}$, then $N(\Gamma, \omega)$ will have boundary vertices $\{1, 2, . . . , 2n\}$, where boundary vertex $\bar i$ is identified with $2i-1$, and
a boundary vertex $2i$ in $N(\Gamma, \omega)$ lies between $\bar i$ and $\overline {i + 1}$. The boundary vertex $2i$ can be identified with the vertex $\tilde i$ used to label dual non-crossing partitions. The planar bipartite network $N(\Gamma, \omega)$ always has boundary vertices of degree $1$.
The interior vertices of $N(\Gamma, \omega)$ are as follows: we have a black interior vertex $b_v$ for each interior vertex $v$ of $\Gamma$, and a black interior vertex $b_F$ for each interior face $F$ of $\Gamma$; we have a white interior vertex $w_e$ placed at the midpoint of each interior edge $e$ of $\Gamma$. For each
vertex $\bar i$, we also make a black interior vertex $b_i$. The edges of $N(\Gamma, \omega)$ are defined as follows: 
\begin{itemize}
\item  if $v$ is a vertex of an edge $e$ in $\Gamma$, then $b_v$ and $w_e$ are joined, and the weight of this edge is equal
to the weight $\omega(e)$ of $e$ in $\Gamma$, 
\item  if $e$ borders $F$, then $w_e$ is joined to $b_F$ by an edge with weight $1$, 
\item  the vertex $b_i$ is joined by an edge with weight $1$ to the boundary vertex $2i - 1$ in $N$, and $b_i$
is also joined by an edge with weight 1 to $w_e$ for any edge $e$ incident
to $\bar i$ in $\Gamma$, 
\item  even boundary vertices $2i$ in $N(\Gamma, \omega)$ are joined by an edge with weight $1$
to the face vertex $w_F$ of the face $F$ that they lie in.
\end{itemize}
\end{definition}
\begin{definition} \label{conc}
We call an $(n-1)$-element subset $I \subset 1 \dots 2n$ {\em concordant} with a non-crossing partition $\sigma$ if each part of $\sigma$ and each part of the dual partition $\tilde{\sigma}$
contains exactly one element not in $I$. In this situation we also say that $\sigma$ or $(\sigma, \tilde{\sigma})$ is concordant with $I$.
\end{definition}

\begin{theorem} \cite{L} \label{elcon}
Let $e \in E_n$ and $N(\Gamma, \omega)$ be the bipartite network associated with $e$ then the following hold: 
$$
\Delta_{I}^M= \sum\limits_{(\sigma, I)} L_{\sigma},
$$
here summation is over all such $\sigma,$ which are concordant with $I.$ 
\end{theorem}

\begin{example}
For the networks on Fig. \ref{fig:grove} we have
$$
L_{1|2|3}=a+b+c,\; L_{123}=abc,\; L_{12|3}=ac,\; L_{13|2} = bc,\; L_{23|1} =ab;
$$
$$
\Delta_{12}=\Delta_{45}=bc,\; \Delta_{23}=\Delta_{56}=ab,\; \Delta_{34}=\Delta_{16}=ac,\; \Delta_{13}=\Delta_{35}=\Delta_{15}=abc,
$$
$$
\Delta_{24}=\Delta_{46}=\Delta_{26}=a+b+c,\; \Delta_{14}=ac+bc,\;\Delta_{25}=ba+bc,\;\Delta_{36}=ab+ac.
$$ 
\end{example}

\begin{figure}[h]
     \centering
     \includegraphics[scale=1]{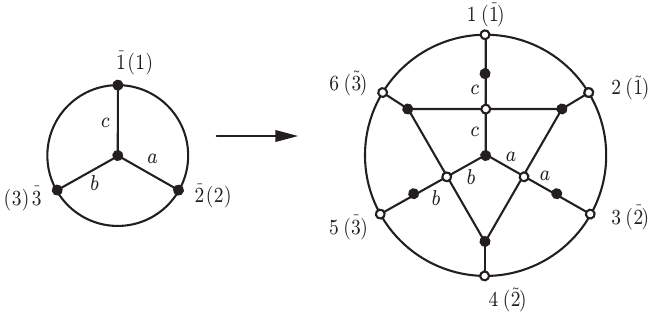}
     \caption{Grove measurements and Plucker coordinates}
     \label{fig:grove}
\end{figure}

Now we establish the connection between the grove measurements $L, L_{ij}, L_{kk}$ and the Plucker coordinates derived out of a bipartite network:
   
 \begin{lemma} \label{lemmal}
 For  $e(\Gamma,\omega)\in E_n$ and  $N(\Gamma, \omega)$ as above the following holds:
 $$L=\Delta_{R}^M,$$ where $R$ is any set of even indices
 such that  $|R|=n-1$.

 $$L_{ij}=\Delta_{\{2i-1\}\cup N}^M,$$ where $N$ is the set of all even indices except of the two closest to  
 $2j-1$ and $|N|=n-2.$ 
 
 $$L_{kk}=\Delta_{\{2k-1\}\cup T}^M$$ where $T$ is the set of all even indices except of the two closest to $2k-1$ and $|T|=n-2$. 
 \end{lemma}
\begin{proof}
The dual for the partition $\sigma=(\bar{1}|\bar{2}|\dots|\bar{n})$ is $\widetilde{\sigma}=(\widetilde{1} \widetilde{2} \dots \widetilde{n}).$ It is easy to see that the set $(\sigma, \widetilde{\sigma})$ is concordant with $R$  and there are no  other $\sigma$ with this property, so $L=\Delta_{R}^M.$

The dual for the partition $\sigma=(\bar{i}\bar{j}|*|\dots|*)$ has the form $\widetilde{\sigma}=(\widetilde{A}\setminus\{\widetilde{j_1}\}|\widetilde{B}\setminus\{\widetilde{j_2}\})$ where $\widetilde{j_1}$ and $\widetilde{j_2}$ are two closest to $\bar{j}$ and $\widetilde{A}, \widetilde{B}$ are subsets of $(\widetilde{1} \widetilde{2} \dots \widetilde{n})$. It is easy to see that $\sigma$ is concordant with  the set $\{2i-1\}\cup N.$
 
 Suppose there is another $\sigma$ concordant with  the set $\{2i- 1\}\cup N$ then $\sigma$ must have the form $\sigma=(\bar{i}\bar{k}|*|\dots|*)$ ($k\neq j$) and its dual has the form $\widetilde{\sigma}=(\widetilde{A'}|\widetilde{B'})$, where $\widetilde{A}', \widetilde{B}'$ are subsets of $(\widetilde{1} \widetilde{2} \dots \widetilde{n})$. Then at least one of the  two closest to  $2j-1$ must belong to $J$, because either $\widetilde{A'}$ or $\widetilde{B'}$ contains both $\widetilde{j}_1$ and $\widetilde{j}_2$, hence we get a contradiction and $L_{ij}=\Delta_{\{2i-1\}\cup N}^M.$

 It is easy to see that each partition of the form $\sigma=(\bar{k}\bar{i}|*|\dots|*)$ is concordant with  the set $\{2k-1\}\cup T.$
 Finally notice that any $\sigma$ which is concordant with the set $\{2k-1\}\cup T$ must have a form $\sigma=(\bar{k} \bar{i}|*|\dots|*)$ ($k \neq i$). On the other hand we have already listed all such  $\sigma$, therefore $L_{kk}=\Delta_{\{2k-1\}\cup T}^M$.
 \end{proof}
 \begin{example}
 For any network from $E_5$ we have $L_{13}=\Delta_{\{1,2,8,10\}}^M$ and $L_{11}=\Delta_{\{1,4,6,8\}}^M.$
 \end{example}
 The following theorem was proved in \cite{KW}:
\begin{theorem}  \label{kenwil}
For an electrical network $e(\Gamma,\omega)\in E_n$  the off-diagonal elements of its response matrix $M_R(e)$ satisfies the following relation: 
 $$x_{ij}=-\frac{L_{ij|*|*|*\ldots}}{L_{*|*|*\ldots}}:=-\frac{L_{ij}}{L}.$$
\end{theorem}

\section{Electrical networks and Grassmannian}
\label{sec:networks}
In this section we define a new parametrisation of the image of $E_n$ in $\mathrm{Gr}(n-1,2n)$ constructed in \cite{L}. This parametrisation is based on the ideas from the Postnikov paper \cite{P} on the parametrisation of the non-negative Grassmannian. 
 
\subsection{New parametrisation}

\begin{definition}
Let $e(\Gamma,\omega)\in E_n$ and $M_R(e)=(x_{
ij})$ denote its response matrix. Define a point $X(e)$ in
$\mathrm{Gr}(n-1,2n)$ associated to $e(\Gamma,\omega)$ as the row space of the matrix:
\bea
\Omega_n(e)=\left(
\begin{array}{cccccccc}
x_{11} & 1 & -x_{12} & 0 & x_{13} & 0 & \cdots & (-1)^n\\
-x_{21} & 1 & x_{22} & 1 & -x_{23} & 0 & \cdots & 0 \\
x_{31} & 0 & -x_{32} & 1 & x_{33} & 1 & \cdots & 0 \\
\vdots & \vdots &  \vdots &  \vdots &  \vdots & \vdots & \ddots & \vdots
\end{array}
\right)
\eea
\end{definition}
\begin{example} 

For $e\in E_4$, $\Omega_4(e)$ has the following form: 
 \begin{equation*}
\Omega_4(e) = \left(
\begin{array}{ccccccccc}
x_{11}& 1  & -x_{12}  & 0 & x_{13} & 0 & -x_{14} & 1 \\
-x_{21}& 1  & x_{22}  & 1 & -x_{23} & 0 & x_{24} & 0 & \\
x_{31}& 0  & -x_{32}  & 1 & x_{33} & 1 & -x_{34} & 0 & \\
-x_{41}& 0  & x_{42}  & 0 & -x_{43} & 1 & x_{44} & 1 & \\
\end{array}
\right),
\end{equation*} 

\end{example}
The sum of the elements in each row in $M_R(e)$ is equal to zero hence the rank of $\Omega_n(e)$ is equal to $n-1$ and it indeed defines a point in $\mathrm{Gr}(n-1,2n)$.

Here is one of the main results of the paper:

\begin{theorem}\label{maint} The row space of \, $\Omega_n(e)$ defines the same point in
$\mathrm{Gr}_{\geq 0}(n-1,2n)$ as the point defined by $e(\Gamma,\omega)$ under the Lam embedding.
\end{theorem}

Let $N(\Gamma, \omega)$ be the bipartite network associated to $e\in E_n$ and  $O(I)$ be a perfect orientation  on $N(\Gamma, \omega)$ with the set of sources $I$ labeled by even indices such that $|I|=n-1$. Such an orientation exists by Lemma \ref{lemmal}:
$$\Delta_{I}^M=\sum\limits_{\Pi \in \Pi(I)}\mathrm{wt}(\Pi)=L \neq 0,$$ 
so by Remark \ref{bij} for each $\Pi \in \Pi(I)$ there is a perfect orientation $O_{\Pi}(I)$.


Our goal is to connect the Plucker coordinates corresponding to the networks $N(\Gamma, \omega', O)$, $N(\Gamma, \omega)$,  and the Postnikov network $PN(\Gamma, \omega',O)$, where the weight function $\omega'$ is defined in \ref{link}. 

We use the following statement from \cite{T}:
\begin{theorem} \cite{T} \label{tal}
Let $PN(\Gamma, \omega , O)$ be a Postnikov network with the set of sources $I$ and $A$ be its extended boundary measurements matrix, then the minors of the matrix $A$ can be calculated as follows:
$$\Delta_J(A)=\frac{\sum_F \mathrm{wt}(F)}{\sum_C \mathrm{wt}(C)},$$
where the sum in the numerator is over all flows whose paths connect vertices in the set of sources  $I$ with vertices in the set $J$. The sum in the denominator is over all flows which do not connect points on the boundary.
\end{theorem}
As a corollary we obtain:
\begin{corollary} \label{contal}
Let $N(\Gamma,\omega)$ be the bipartite network associated to $(\Gamma,\omega)\in E_n$ and $O(I)$ be a perfect orientation with the set of sources $I$ labeled by even indices. For a Postnikov network $PN(\Gamma, \omega',O)$ the minors of the extended boundary measurement matrix $A$ obey
$$\Delta_J(A)=\frac{\Delta_J^M}{L},$$
where $L=\Delta_I^M$ is a grove measurement as in Lemma \ref{lemmal}. 
\end{corollary}
\begin{proof}
According to Theorems \ref{link} and \ref{tal} we have:   
$$\Delta_J(A)=\frac{\sum_F \mathrm{wt}(F)}{\sum_C \mathrm{wt}(C)}=\frac{\Delta_J^F}{\sum_C
\mathrm{wt}(C)}=\frac{\Delta_J^M}{\mathrm{wt}(\Pi_0)\sum_C \mathrm{wt}(C)},$$
where $\mathrm{wt}(\Pi_0)$ is the weight of the almost perfect matching corresponding to the orientation $O(I)$ as explained in Remark \ref{bij}. On the one hand, Definition \ref{matrpost} implies that $\Delta_I(A)=1$, on the other hand we have:
$$\Delta_I(A)=\frac{\sum_F \mathrm{wt}(F)}{\sum_C \mathrm{wt}(C)}=\frac{\Delta_I^F}{\sum_C \mathrm{wt}(C)}=\frac{\Delta_I^M}{\mathrm{wt}(\Pi_0)\sum_C \mathrm{wt}(C)},$$
therefore $\Delta_I^M=\mathrm{wt}(\Pi_0)\sum \limits_C \mathrm{wt}(C)$. To finish the proof we use Lemma \ref{lemmal} which implies that 
$$L=\Delta_I^M=\mathrm{wt}(\Pi_0)\sum_C \mathrm{wt}(C).$$ 
\end{proof}

Now we are ready to prove our theorem.  Fix an electrical network $e(\Gamma,\omega)\in E_n$ and an arbitrary $k \in [1\dots n]$. Let $I_k$  be the set of all even indices from $[2 \dots 2n]$ except the  even index closest to $2k-1$ {\em counterclockwise}. We will denote it by $j_k$. It is clear that $|I_k|=n-1$. 
Denote by  $j^k$ the even index closest to $2k-1$ {\em clockwise}.

For $e(\Gamma,\omega)$ and $N(\Gamma,\omega)$ be as above let  $PN(\Gamma,\omega',O_k)$ 
be the Postnikov network where $O_k$ is a perfect orientation with the set of sources  $I_k$. Let $A_k$ be the extended boundary measurement matrix of $PN(\Gamma,\omega',O_k)$, see Definition \ref{def:extendedmatr}. 
\begin{theorem} \label{constr}
Denote by $a_{j^k \to m}$ the entry of $A_k$ which corresponds to the weight of the flow
 from $j^k$ to $m.$ Then the following holds:
\\
1. $a_{j^k \to m}=(-1)^{s+1}x_{k\frac{m+1}{2}}$ for any sink with an odd index $m \neq 2k-1$ where $s$ is the number of the elements in $I_k$ situated between the source labeled by $j^k$ and the sink labeled by $m$;\\
2. $a_{j^k \to m}=x_{kk}$ for the sink labeled by $m = 2k-1$;\\
3. $a_{j^k \to m}=(-1)^s$ for the sink labeled by $m = j_k$ where $s$ is the number of the elements in $I_k$ situated between the source labeled by $j^k$ and the sink labeled by $j_k$;\\
4. $a_{j^k \to j^k}=1$; \\
5. $a_{j^k \to m}=0$ if there is a source labeled by $m \neq j^k$.
\end{theorem}

\begin{figure}[h]
     \centering
     \includegraphics[scale=0.8]{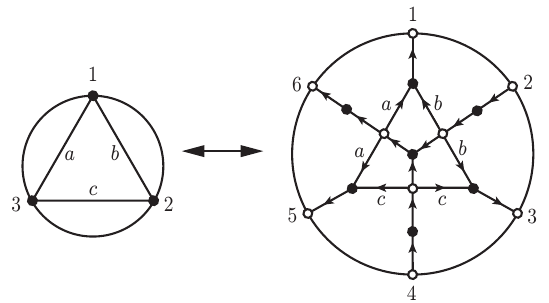}
     \caption{Example \ref{ex}}
\end{figure}

\begin{example}
\label{ex}
$$L_{1|2|3} = 1,\; L_{12|3}=b,\; L_{13|2}=a,\; L_{23|1} = c;$$
$$a_{21} = a+b,\; a_{22} = 1,\; a_{23}=b,\;a_{24} = 0,\; a_{25} = -a,\; a_{26}=-1;$$
$$x_{12} = \frac{-L_{12|3}}{L_{1|2|3}} = -b,\; x_{13} = \frac{-L_{13|2}}{L_{1|2|3}} = -a.$$
\end{example}

\begin{proof}

\textcolor{black}{1.} It follows from Definition \ref{matrpost} that
$$\Delta_{\{m\}\cup I_k -\{j^k\} }(A_k)=(-1)^{s} a_{j^k \to m}.$$
On the other hand Corollary \ref{contal} implies that:
$$\Delta_{\{m\}\cup I_k -\{j^k\} }(A_k)=\frac{\Delta_{\{m\}\cup I_k -\{j^k\} }^M}{L}.$$
Lemma \ref{lemmal} and Theorem  \ref{kenwil} now prove the claim:
$$\frac{\Delta_{\{m\}\cup I_k -\{j^k\} }^M}{L}=\frac{L_{\frac{m+1}{2}k}}{L}=-x_{\frac{m+1}{2}k}=-x_{k\frac{m+1}{2}}.$$

\textcolor{black}{2.} As it follows from Definition \ref{matrpost}
$$\Delta_{\{m\}\cup I_k -\{j^k\} }(A_k)=a_{j^k \to m}.$$
On the other hand Corollary \ref{contal} implies that
$$\Delta_{\{m\}\cup I_k -\{j^k\} }(A_k)=\frac{\Delta_{\{m\}\cup I_k -\{j^k\} }^M}{L}.$$
Lemma \ref{lemmal} and Theorem \ref{kenwil} now prove the claim:
$$\frac{\Delta_{\{m\}\cup I_k -\{j^k\} }^M}{L}=\frac{L_{kk}}{L}=x_{kk}.$$

\textcolor{black}{3.}  As it follows from Definition \ref{matrpost}
$$\Delta_{\{m\}\cup I_k -\{j^k\} }(A_k)=(-1)^{s}a_{j^k\to m}.$$
 On the other hand Corollary \ref{contal} implies that
$$\Delta_{\{m\}\cup I_k -\{j^k\} }(A_k)=\frac{\Delta_{\{m\}\cup I_k -\{j^k\} }^M}{L}.$$
Lemma \ref{lemmal} now proves the claim:
$$\frac{\Delta_{\{m\}\cup I_k -\{j^k\} }^M}{L}=\frac{L}{L}=1.$$

\textcolor{black}{4.} Since $m=j^k$ is the source the claim
immediately follows from Definition \ref{matrpost}.

\textcolor{black}{5.} Since $m \neq j^k$ is the source the claim 
immediately follows from Definition \ref{matrpost}.

\end{proof}

It remains to notice that by Theorem \ref{chanorient} all the vectors representing the rows of  the matrix $\Omega_n(e)$ belong to the same subspace defined by the extended boundary measurements matrix $A_k$ for any $k$.
Moreover by Corollary \ref{contal} this subspace defines the same point in $\mathrm{Gr}(n-1,2n)$  as the point defined by the bipartite network $N(\Gamma,\omega)$, and by Theorem  \ref{elcon} this point belongs to $\mathrm{Gr}(n-1, 2n)_{\geq 0}$. This ends the proof of Theorem \ref{maint}.

\subsection{Compactification}\label{compactification}
In this section we extend our parametrisation  of $E_n$ to its compactification introduced in \cite{L}. Let $\overline{E}_n$ be the set of electrical networks  the weights of some edges of $\Gamma$ are allowed to be equal to $0$ or $\infty$. 

One can give a combinatorial interpretation of 
zero and infinite values of conductivity. Indeed by the Kirchhoff law the zero conductivity of an edge means the deletion of this edge, the infinite conductivity of an edge means the contraction of this edge. 

The edges connecting two boundary vertices can have conductivity equal to $\infty$ as well. It leads us to the definition of a {\em cactus network}. 

\begin{definition}\label{def:comp}
Let $S$ be a circle with $n$ boundary points labeled by $1, \dots, n$. Let $\sigma\in \mathcal{NC}_n$, identifying the boundary points according to the parts of $\sigma$ gives a hollow cactus $S_\sigma$. It is a union of circles glued together at the identified points. The interior of $S_\sigma$, together with $S_\sigma$ itself is called a {\em cactus}. A {\em cactus network} is a weighted graph $\Gamma$ embedded into a cactus. In other words, we might think about  cactus networks as the networks which are obtained by contracting the set of edges between the boundary vertices  defined by $\sigma\in \mathcal{NC}_n$ which have infinite conductivity.
\end{definition}

Let  $e(\Gamma, \omega) \in \overline E_n$ have no edges with infinite conductivity which connect boundary vertices. After deleting and contracting all edges with the conductivities equal to zero and infinity respectively we obtain a new electrical network $e'(\Gamma', \omega')$. It is easy to see that for $e'$ the response matrix $M_R(e')$ is well-defined therefore we are able to use Theorem \ref{maint} to construct the matrix  $\Omega_n(e')$. However if a network $e$ contains  edges with infinite conductivities between boundary vertices we will not be able to construct the matrix $M_R(e)$ and therefore $\Omega_n(e')$. In order to resolve  this problem we have to slightly change the parametrisation defined by  Theorem \ref{maint}.
\begin{definition}
We let $R_{i,j}$ denote the \textit{effective electrical resistance} between nodes $i$ and $j$ i.e. the voltage at node $i$ which, when node $j$ is held at $0$ volts, causes a unit current to flow through the circuit from node $i$ to node $j$.
\end{definition}
\begin{theorem}\cite{KW} \label{kw-rep} Let $e(\Gamma, \omega)\in E_n$ be a planar circular electrical network on a connected graph $\Gamma$ and $M_R(e)$ is its response matrix. Denote by $M'_R(e)$ the matrix obtained from $M_R(e)$ by deleting the first row and the last column, then the following holds
 \begin{itemize}
 \item $M'_R(e)$ is invertible; 
 \item The matrix elements of its inverse are given by the formula
 \begin{equation}
M'_R(e)^{-1}_{ij}=\begin{cases}
   R_{in},\, \text{if}\,\, i=j   \\
   \frac{1}{2}(R_{in}+R_{jn}-R_{ij}),\, \text{if}\,\, i\not = j,\\
    
      \end{cases}
    \end{equation} 
\end{itemize}
where $R_{ij}$ is the effective resistance between the boundary points $i$ and $j$.
\end{theorem}


Let $e(\Gamma, \omega)$ be as in Theorem \ref{kw-rep}.  Denote by $\Omega'_n(e)$ the matrix obtained from $\Omega_n(e)$ by deleting the last row. Note it has the same row space as $\Omega_n(e)$. Assign to $\Omega'_n(e)$ the following matrix
\bea\label{eq:compact}
M'_R(e)^{-1}D_{n-1}\Omega'_n(e),
\eea
where $D_{n-1}$ is a diagonal $n-1 \times n-1$ matrix with the entries $d_{ii}=(-1)^{i+1}.$  The matrix \eqref{eq:compact} and $\Omega_n(e)$ define the same point of $\mathrm{Gr}(n-1,2n).$

Using \ref{kw-rep} and a particular form of \eqref{eq:compact} we conclude:
\begin{corollary} \label{norm}  The matrix entries of \eqref{eq:compact} are linear combinations of $R_{ij}$'s.
\end{corollary}
\begin{proof}
 For all columns with even indexes of the matrix $M'_R(e)^{-1}D_{n-1}\Omega'_n(e)$  the proposition is obvious, because the columns with even indexes of the matrix $D_{n-1}\Omega'_n(e)$ contain only entries equal to $0, 1$ or $-1.$ 
 
For all the columns except the one indexed by $2n-1$ the statement follows from Theorem \ref{kw-rep}. To finish the proof notice  that  due to Theorem \ref{respc} the column of $D_{n-1}\Omega'_n(e)$ indexed by $2n-1$ is a linear combination of other odd columns with constant coefficients.

\end{proof}

If the boundary  vertices $i$ and $j$ are connected by an edge with infinite conductivity, then $R_{ij}=0$, thus
we obtain the algorithm for finding the parametrization of the points of $\mathrm{Gr}(n-1,2n)$ associated with the cactus networks given below.

As it is explained in Definition \ref{def:comp} an arbitrary cactus network $\overline{e}$ is obtained from a network $e(\Gamma, \omega)$ by contracting some edges $e_{ij}$ between the boundary vertices which have the infinite conductivity. Produce a new network $e_{aux}$ by changing the conductivities of $e_{ij}$ to finite quantities $c_{ij}$. For simplicity we assume at first that $\Gamma$ is connected. The matrix $\Omega'_n(e_{aux}) $ defines the same point of $\mathrm{Gr}(n-1,2n)$ as
$$M'_R(e_{aux})^{-1}D_{n-1}\Omega'_n(e_{aux}).$$
Passing to the limit when $R_{ij} \to 0$ is equivalent to passing to the limit when $c_{ij}\to \infty$ therefore we conclude that the cactus network $\overline{e}$ defines the following point of $\mathrm{Gr}(n-1,2n)$:

\begin{equation}
   X(\overline{e})=\lim_{c_{ij}\to \infty}\left(M'_R(e_{aux})^{-1}D_{n-1}\Omega'_n(e_{aux})\right)=\lim_{R_{ij}\to 0}\left(M'_R(e_{aux})^{-1}D_{n-1}\Omega'_n(e_{aux})\right).
\end{equation}

\begin{remark}
For a cactus network $\overline{e}$ which is obtained from a network $e(\Gamma, \omega)\in E_n$ with $\Gamma$
being not connected we have to apply the algorithm to each connected component.
  
\end{remark}
\section{Lagrangian Grassmannian}
\label{sec:Lagr}
Recall a {\it symplectic vector space} is a vector space equipped with a symplectic bilenear form.
An {\it isotropic subspace} of a symplectic vector space is a vector subspace on which the symplectic form vanishes.
A maximal isotropic subspace is called  {\it Lagrangian} subspace.
For $V$ a symplectic vector space, the {\it Lagrangian Grassmannian} $\mathrm{LG}(n,V)$ is the space of lagrangian subspaces. If it is clear which symplectic space we are working with we will denote the Lagrangian Grassmannian $\mathrm{LG}(n)$.
 
In this section we obtain an important consequence of our parametrisation of the set $E_n$, namely we will prove in Theorem \ref{laggr} that the matrix $\Omega_n(e)$ defines a point in the Lagrangian Grassmannian $\mathrm{LG}(n-1)$. 

\begin{definition} 
Define the subspace
\begin{equation*}
V:=\{v\in \Bbb C^{2n}| \sum_{i=1}^n(-1)^i v_{2i}=0\,, \sum_{i=1}^n (-1)^iv_{2i-1}=0\}.    
\end{equation*}
\end{definition}

\begin{lemma}
All the rows of the matrix  $\Omega_n(e)$ belong to the subspace $V$.
\end{lemma}
\begin{proof}
The proof follows from the third property of the matrix $M_R(e)$ given in Theorem \ref{respc}.
\end{proof}

Fix a basis for the subspace $V$:
\begin{multline}
\label{V:basis}
w_1=(1, 0, 1, 0,  \dots, 0, 0, 0),\; w_2=(0, 1, 0, 1, \dots, 0, 0, 0),\dots,\\
w_{2n-2}=(0, 0, 0, 0, \dots, 1, 0, 1).
\end{multline}

Expanding the rows of the matrix $\Omega_n(e)$ in this basis we obtain the main object of this section -- the matrix $\widetilde{\Omega}_n(e)$ defined as
\begin{equation*}\label{expansion}
\widetilde{\Omega}_n(e)=\Omega_n(e)B_n^{-1},   
\end{equation*}
here  $B_n$ is the matrix whose rows are the vectors $w_i$ and $B_n^{-1}$ is its right inverse:
\begin{equation}
\label{eq:Binverse}
B_n^{-1} = \left(
\begin{array}{cccccc}
1& 0  & -1  & 0 & \dots&0  \\
0& 1  & 0  & -1 & \dots&(-1)^{n+1}\\
0& 0  & 1  & 0 & \dots&0\\
0&0&0&1&\dots&(-1)^n\\
\vdots&\vdots&\vdots&\vdots&\ddots&\vdots\\
0& 0  & 0  & 0 & \dots&1\\
0& 0  & 0  & 0 & \dots&0\\
0& 0  & 0  & 0 & \dots&0
\end{array}
\right).
\end{equation} 

\begin{example} \label{exlagr}
For  $e\in E_3$,  $\widetilde{\Omega}_3(e)$ has the following form: 
 \begin{equation*}
\widetilde{\Omega}_3(e) = \left(
\begin{array}{ccccc}
x_{11}& 1  & x_{13}  & -1 & \\
-x_{21}& 1  & -x_{23}  & 0 & \\
x_{31}& 0  & -x_{31}-x_{32}  & 1 &  \\
\end{array}
\right).
\end{equation*} 

For $e\in E_4$, $\widetilde{\Omega}_4(e)$ has the following form: 
 \begin{equation*}
\widetilde{\Omega}_4(e) = \left(
\begin{array}{ccccccc}
x_{11}& 1  & x_{13}+x_{14}  & -1 & -x_{14} & 1 & \\
-x_{21}& 1  & -x_{23}-x_{24}  & 0 & x_{24} & 0 & \\
x_{31}& 0  & -x_{31}-x_{32}  & 1 & -x_{34} & 0 &  \\
-x_{41}& 0  & x_{41}+x_{42}  & 0 & -x_{41}-x_{42}-x_{43} & 1 &  \\
\end{array}
\right),
\end{equation*} 

\end{example}
Now we are ready to formulate the main result of this section. Let $\Lambda_{2n-2}$ be the following symplectic form on the space $\Bbb C^{2n-2}$
\begin{equation} \label{def:lambda}
\Lambda_{2n-2} = 
\left(
\begin{array}{cccccccccc}
0& 1  & 0  & 0    &\ldots  & 0& \\
-1& 0  &-1  & 0   &\ldots&   0& \\
0& 1  & 0  & 1  &\ldots&   0& \\
\vdots & \vdots & \ddots & \vdots  & \ddots &   \vdots & \\
\end{array}
\right),
\end{equation} 
then the following theorem holds:
\begin{theorem} \label{laggr} For $e\in E_n$ the matrix $\widetilde{\Omega}_n(e)$ defines a point of $\mathrm{LG}(n-1,V).$ 
In other words the following identity holds:
\bea
\widetilde{\Omega}_n(e)\Lambda_{2n-2}\widetilde{\Omega}_n(e)^T=0.
\eea
\end{theorem}

We will prove some auxiliary statements first which will be also useful when we will treat the vertex representation of electrical network.

Let $\Bbb R^{2n}$ be a vector space equipped with the symplectic form 
\bea
\lambda_{2n}=\left(
\begin{array}{cc}
0 & g\\
-g^T & 0
\end{array}
\right);
\eea
where
\bea
g=\left(
\begin{array}{ccccc}
1 & 0 & 0 & \cdots & 0\\
1 & 1 & 0 & \cdots & 0\\
1 & 1 & 1 & \cdots & 0\\
\vdots & \vdots & \vdots & \ddots & \vdots\\
1 & 1 & 1 & \cdots & 1
\end{array}
\right).
\eea

Let $V'\subset \Bbb R^{2n}$ of dimension $2n-2$ defined as follows
$$V':=\{v\in \Bbb R^{2n}| \sum_{i=1}^n v_{i}=0\;\; \text{and}\, \sum_{i=n+1}^{2n} v_{i}=0\}$$

Then $2n-2$ vectors 
\bea
\{e_1-e_2,e_2-e_3,\ldots ,e_{n+1}-e_{n+2},\ldots ,e_{2n-1}-e_{2n}\}\label{basis-V'}
\eea
form a basis for $V'$ and together with $e_{1},e_{n+1}$ they form a basis of $\Bbb R^{2n}$. 

\begin{lemma} 
The restriction of $\lambda_{2n}$ to $V'$ in the basis \eqref{basis-V'} of $V'$  takes the form:
\bea
\left(
\begin{array}{cc}
0 & -h^T\\
h & 0
\end{array}
\right);
\eea
where
\bea
h=\left(
\begin{array}{ccccc}
-1 & 0 & 0 & \cdots & 0\\
1 &- 1 & 0 & \cdots & 0\\
0 & 1 & -1 & \cdots & 0\\
\vdots & \vdots & \vdots & \ddots & \vdots\\
0 & 0 & \cdots & 1&-1
\end{array}
\right)
\eea
and is \textbf{non-degenerate}.
\end{lemma}
Let $M_R(e)=(x_{ij})$ be the response matrix for $e\in E_n$.  Denote by $\Omega^{aux}(e)$ the subspace of $V'$ of dimension $n-1$ spanned by the $n$ vectors 
$$(1,0\ldots ,0,-1,x_{11},\ldots ,x_{1n})$$
$$(-1,1,0\ldots ,0,x_{21},\ldots ,x_{2n})$$
$$(0,-1,1,0\ldots ,0,x_{31},\ldots ,x_{3n})$$
$$\ldots$$
$$(0,\ldots ,0-1,1,x_{n1},\ldots ,x_{nn}).$$
We denote by the same symbol $\Omega^{aux}(e)$ the matrix with these vectors as rows.

\begin{lemma} \label{vertex model} The restriction of $\lambda_{2n}$ to the row space of $\Omega^{aux}(e)$ is zero for any value of the parameters therefore this subspace is a point in $\mathrm{LG}(n-1,V')$. 
\end{lemma}
\begin{proof}
Indeed the sum of the elements in each row of the response matrix $M_R(e)$ is equal to zero, therefore: 
\begin{equation*}
\Omega^{aux}(e)\lambda_{2n}\Omega^{aux}(e)^t=
\left(\begin{array}{ccccc}
0 & x_{21}-x_{12} & x_{31}-x_{13} &  \cdots & x_{n1}-x_{1n}\\
x_{12}-x_{21}& 0 & x_{32}-x_{23}  & \cdots & x_{n2}-x_{2n}\\
x_{23}-x_{32}& x_{34}-x_{43} & 0 &   \cdots & x_{n3}-x_{3n}\\
\vdots & \vdots & \vdots & \ddots & \vdots\\
\end{array}
\right).
\end{equation*}
And $M_R(e)$ is a symmetric matrix, so we obtain $\Omega^{aux}(e)\lambda_{2n}\Omega^{aux}(e)^t=0$.
\end{proof}

\begin{lemma}\label{vertex model 2}
\begin{equation}
    \Omega_n^{aux}(e)=D_{n}\Omega_n(e)\overline{T}_{2n},
\end{equation}

Where $\overline{T}_{2n}$ is a $2n\times 2n$ signed permutation matrix corresponding to the permutation in one-row notation
\[(2,-4,6,\ldots,(-1)^{n+1}2n,1,-3,5,\ldots,(-1)^{n+1}(2n-1)),\]
or as a matrix:
\begin{equation*}
\overline{T}_{2n}=\left(
\begin{array}{cccccccc}
0 & 0 & 0 &\cdots & 1 & 0 &0 &\cdots\\
1 & 0 & 0 & \cdots & 0&0 & 0 & \cdots \\
0 & 0 & 0 & \cdots & 0 & -1 & 0 &\cdots \\
0 & -1 & 0 & \cdots & 0 & 0 & 0 &\cdots \\
0 & 0 & 0 & \cdots & 0 & 0 & 1 &\cdots \\
0 & 0 & 1 & \cdots & 0 & 0 &0  &\cdots \\
\vdots &  \vdots & \vdots & \ddots & \vdots & \vdots & \vdots
\end{array}
\right).
\end{equation*}
 
\end{lemma}
\begin{proof}
The proof follows immediately from the structure of the matrix $\Omega_n(e)$. 
\end{proof}
Putting together the facts proved above we obtain

\[0 = \Omega_n^{aux}(e)\lambda_{2n}(\Omega_n^{aux}(e))^t=D_{n}\Omega_n(e)\overline{T}_{2n}\lambda_{2n}\overline{T}^t_{2n}(\Omega_n(e))^tD_{n}^t=\]
\[D_{n}\widetilde{\Omega}_n(e)B_n\overline{T}_{2n}\lambda_{2n}\overline{T}^t_{2n}B^t_n(\widetilde{\Omega}_n(e))^tD_{n}^t\]    
Since $D_{n}$ is invertible
\bea
\widetilde{\Omega}_n(e)B_n\overline{T}_{2n}\lambda_{2n}\overline{T}^t_{2n}B^t_n\widetilde{\Omega}_n(e)^t=0
\eea

Finally a direct calculation shows that
\bea
B_n\overline{T}_{2n}\lambda_{2n}\overline{T}^t_{2n}B^t_n=\Lambda_{2n-2}, 
\eea
which finishes the proof of the Theorem \ref{laggr}.
\begin{remark}
Thus, we obtained  that  every $e\in E_n$ defines the point of $\mathrm{LG}(n-1,V).$ Using the continuity argument described in Section \ref{compactification} we immediately  conclude that every $e\in \overline{E}_n$ defines the point of $\mathrm{LG}(n-1,V)$ as well. From now on we will use the notation $\mathrm{LG}(n-1)$ instead of $\mathrm{LG}(n-1,V)$.
\end{remark}

\section{Fundamental representation of the symplectic group} 
\label{sec:representation}
 In this section we connect the results from \cite{L} to representation theory of the symplectic group.
 The symplectic group is a classical group defined as the set of linear transformations of a $2n$-dimensional vector space over $\Bbb C$ which preserve a non-degenerate skew-symmetric bilinear form. Such a vector space is called a symplectic vector space, and the symplectic group of an abstract symplectic vector space $V$ is denoted $\mathrm{Sp}(V)$. Upon fixing a basis for $V$, the symplectic group becomes the group of $2n \times 2n$ symplectic matrices under the operation of matrix multiplication. This group is denoted by $\mathrm{Sp}(2n,\Bbb C)$. If the bilinear form is represented by the nonsingular skew-symmetric matrix $\Lambda$, then
\[
{\displaystyle \operatorname {Sp} (2n,\Bbb C)=\{M\in M_{2n\times 2n}(\Bbb C):M^{\mathrm {T} }\Lambda M=\Lambda \},}\]
where $M^T$ is the transpose of $M$. 

The Lie algebra of $\mathrm{Sp}(2n, \Bbb C)$ is the set
\[{\displaystyle {\mathfrak {sp}}(2n,\Bbb C)=\{X\in M_{2n\times 2n}(\Bbb C):\Lambda X+X^{\mathrm {T} }\Lambda =0\},}\]
equipped with the commutator as its Lie bracket.

\subsection{The Lam group}
 
  \begin{definition}
  \label{def:xy}
Introduce some auxiliary $2n \times 2n$ matrices
\begin{itemize}
    \item if $i \in 1, \dots, 2n-1$ then $x_i(t)$ is the upper-triangular matrix with ones on the main diagonal and only one non-zero entry $(x_i)_{i,
    i+1}=t$ above the diagonal
    \item if $i \in 1, \dots, 2n-1$ then $y_i(t)$ is the low-triangular matrix  with ones on the main diagonal and only one non-zero entry $(y_i)_{i+1,i}
    =t$ below the diagonal 
    \item if $i =2n$ then 
    $x_{2n}(t)=s_{n} x_1(t) s_{n}^{-1}$ 
     \item if $i=2n$ then $
y_{2n}(t)=s_{n} y_1(t) s_{n}^{-1}$
\end{itemize}  
    where  
    \begin{equation} \label{transpos}
s_{n} = \left(
\begin{array}{cccccccc}
0 & 1 & 0 & 0 &   \cdots & 0\\
0 & 0 & 1 & 0 &   \cdots & 0 \\
0 & 0 & 0 & 1 &   \cdots & 0 \\
\vdots & \vdots &  \vdots &  \vdots   & \ddots & 1\\
(-1)^{n} & 0 & 0 & 0 &  \cdots & 0 
\end{array}
\right)
\end{equation} 
 \end{definition}
The following proposition appeared in \cite{L} and was proved by T.Lam and A.Postnikov.
\begin{proposition}\label{u} For $1 \leq i \leq 2n$, let $u_i(t) = x_i(t)y_{i-1}(t) = y_{i-1}(t)x_i(t)$, where the indices are taking $mod\,\,2n$, be a one-parameter subgroup of $GL(2n).$
Then the matrices $u_i(t)$ satisfy the following relations
\begin{itemize}
\item $u_i(t_1)u_i(t_2) = u_i(t_1 + t_2)$
\item $u_i(t_1)u_j (t_2) = u_j (t_2)u_i(t_1)$ for $|i - j| \geq 2$
\item
$u_i(t_1)u_{i\pm 1}(t_2)u_i(t_3) = u_{i\pm 1}(t_2t_3/(t_1 + t_3 + t_1t_2t_3))u_i(t_1 + t_3 + t_1t_2t_3)u_{i\pm 1}(t_1t_2/(t_1 + t_3 + t_1t_2t_3))$
\end{itemize}
The set of all matrices $u_i(t)$, where $t$ is an arbitrary parameter and  $i=1, \dots, 2n$, generates the group $\mathcal El_{2n}$ which we will call the Lam group.
\end{proposition}

The Lie algebra $\mathfrak{el}_{2n}$ of $\mathcal El_{2n}$ is generated by definition by the logarithms of the generators $u_i(t)$. Denote these generators $\mathfrak{u}_i$.

An explicit representation of $\mathfrak{el}_{2n}$ as a subalgebra of $\mathfrak sl_{2n}$ was given in \cite{L}: 
$$\mathfrak{u}_i=e_i+f_{i-1},\,\,i=2,...,2n-2$$
$$\mathfrak{u}_1=e_1+(-1)^n[e_{1},[e_{2},[...[e_{2n-1}]]...]$$
$$\mathfrak{u}_0 = \mathfrak{u}_{2n}=f_{2n-1}+(-1)^n[f_{2n-1},[f_{2n-2},[...[f_{1}]]...]$$
where $e_i = E_{i,i+1},\; f_i = E_{i+1,i}$. 
It is easy to check $\mathfrak u_i$  satisfy the relations 
\[
 [\mathfrak{u}_i,\mathfrak{u}_j]=0,\,\text{\rm if}\,\, |i-j|>1\]
\[ [\mathfrak{u}_i,[\mathfrak{u}_i,\mathfrak{u}_j]]=-2\mathfrak{u}_i, \,\text{\rm if}\,\, |i-j|=1\,\text{mod\,\,2n}.
\]

The affine version of the electrical Lie algebra $\hat{\mathfrak{el}}_{2n}$ was introduced in \cite{BGG}. This is the affine version of the object introduced in \cite{LP}. It is defined there together with an embedding into affine version of the algebra $\mathfrak{sl}_{2n}$. The algebra $\hat{\mathfrak{el}}_{2n}$ has finite dimensional representations. 
The image of one of them has generators $\mathfrak{u}_i$ obeying the above relations. 

The group generated by the formal exponents of the elements of the affine electric Lie algebra we will call the {\it affine electric Lie group}.
The Lam group $\mathcal El_{2n}$ is a representation of the affine electric Lie group. 
\begin{remark}
The connection of affine electric Lie algebra and the affine symplectic Lie algebra is not clear to us at the moment.
\end{remark}

We write the operators $x_i(t), y_i(t)$ and $u_i(t)$ as acting on  $\mathrm{Gr}(n-1,2n)$ on the {\bf right} to make it consistent with Section \ref{Sec:nonneg} where these operators  arise in the process of adding spikes and bridges to a bipartite network. We therefore will be using an equivalent definition of the symplectic group and the symplectic algebra. For us, for example, the symplectic group is 
\[
{\displaystyle \operatorname {Sp} (2n,\Bbb C)=\{M\in M_{2n\times 2n}(\Bbb C):M \Lambda M^{\mathrm {T}}=\Lambda \}},\]

Recall that we have defined  the subspace $V\subset \mathbb{C}^{2n}$ of dimension $2n-2$ as follows
$$V:=\{v\in \Bbb C^{2n}| \sum_{i=1}^n(-1)^i v_{2i}=0\, \\, \sum_{i=1}^n (-1)^iv_{2i-1}=0\}.$$
Now, we are able to prove the main result of this section.
\begin{theorem} \label{res}
Operators $u_i(t)|_V$ preserve the symplectic form $\Lambda_{2n-2}$ defined in \eqref{def:lambda}. Moreover the restriction of the Lam group to the subspace $V$ generated by the set of matrices $u_i(t)|_V, i=1, \dots, 2n$ is the representation of the symplectic group $\mathrm{Sp}(2n-2)$.  
\end{theorem}
\begin{proof}

Using the definition of the matrix $B_n$ (see \eqref{eq:Binverse} and above) an explicit form of the restriction of $u_i(t)$ on the subspace $V$ in the basis \eqref{V:basis} could be computed as:
\begin{equation}
\label{uib}
    u_i(t)|_V=B_nu_i(t)B_n^{-1}.
\end{equation}
In particular,

\begin{eqnarray} \label{expl-rest}
u_1(t)|_V =& (s_{n}|_V)^2 (u_3(t)|_V) (s_{n}|_V)^{-2},\\\nonumber
u_2(t)|_V =&y_1(t),\\\nonumber
u_i(t)|_V =&x_{i-2}(t)y_{i-1}(t),\;3\leq i\leq 2n-2,\\\nonumber
u_{2n-1}(t)|_V =&x_{2n-3}(t);\\\nonumber
u_{2n}(t)|_V =& (s_{n}|_V)^2 (u_2(t)|_V)(s_{n}|_V)^{-2}.\nonumber
\end{eqnarray}
here
\begin{equation} 
s_{n}|_V = B_ns_{n}B_n^{-1}= \left(
\begin{array}{cccccccccc}
0& 1  & 0  & 0   &\ldots  & 0&  0 & 0& \\
0& 0  & 1  & 0   &\ldots&   0&  0 & 0&\\
0& 0 & 0  & 1 &\ldots&   0&  0 & 0&\\
\vdots & \vdots & \vdots & \vdots  & \ddots &   \vdots & \vdots &   \vdots & \\
0& 0 & 0  & 0  &\ldots&    0&  0&  1& \\
(-1)^n & 0 & (-1)^{n+1}  & 0  &\ldots&  0 &   1 & 0& \\
\end{array}
\right).
\end{equation}

Let us give here explicit expressions for $u_1(t)|_V$ and $u_{2n}(t)|_V$.

\begin{equation*} 
u_{1}(t)|_V = \left(
\begin{array}{cccccccccc}
1& t  & 0  & -t   &\ldots  & (-1)^{n+1}t&  0 & (-1)^{n}t& \\
0& 1  & 0  & 0   &\ldots&   0&  0 & 0&\\
0& 0 & 1  & 0 &\ldots&   0&  0 & 0&\\
0& 0 & 0  & 1  &\ldots&   0&  0 & 0&\\
0& 0 & 0  & 0  &\ldots&   0&  0 & 0&\\
\vdots & \vdots & \vdots & \vdots  & \ddots &   \vdots & \vdots &   \vdots & \\
0& 0 & 0  & 0  &\ldots&    0&  1&  0& \\
0& 0 & 0  & 0  &\ldots&  0 &   0 & 1& \\
\end{array}
\right),
\end{equation*}
\begin{equation*} 
u_{2n}(t)|_V = \left(
\begin{array}{cccccccccc}
1& 0  & 0  & 0   &\ldots  & 0&  0 & 0& \\
0& 1  & 0  & 0   &\ldots&   0&  0 & 0&\\
0& 0 & 1  & 0 &\ldots&   0&  0 & 0&\\
0& 0 & 0  & 1  &\ldots&   0&  0 & 0&\\
0& 0 & 0  & 0  &\ldots&   0&  0 & 0&\\
\vdots & \vdots & \vdots & \vdots  & \ddots &   \vdots & \vdots &   \vdots & \\
0& 0 & 0  & 0  &\ldots&    0&  1&  0& \\
(-1)^n t& 0 & (-1)^{n+1}t  & 0  &\ldots&  0 &   t & 1& \\
\end{array}
\right).
\end{equation*}

With these formulas it is immediately obvious that
$u_i(t)|_V\Lambda_{2n-2}u_i(t)^t|_V=\Lambda_{2n-2}$.

It is well known that for a simple algebraic group the image of the exponential map generates the whole group. 
The generators $u_i(t)$, $2\leq i\leq 2n-1$ are the exponents of the elements $e_{i-2}+f_{i-1}$, where $f_{2n-2}=e_0=
0$ of the Lie algebra $\mathfrak{sl}_{2n-2}$. These generate the symplectic Lie algebra $\mathfrak {sp}_{2n-2}$ as the dimension count shows therefore we conclude that the above representation of the Lam group
factors through the vector representation of $\mathrm{Sp}(2n-2)$.
\end{proof}

\subsection{The symplectic group and the Lagrangian Grassmannian}

Recall that in \cite{L} it was proved that the image of the set  $\overline E_n$  can be described in combinatorial terms as follows.

\begin{definition}
  Define a matrix $A_n=(a_{I\sigma}),$ where $I$ is a subset of the set $\{1,\dots, 2n\}$, $|I|=n-1$, and $\sigma\in \mathcal{NC}_n$  as follows:
  \begin{equation*}
a_{I\sigma}=\begin{cases}
   1,\, \text{if}\,\, \text{$\sigma$ is concordant with  $I$}   \\
   0,\, \text{otherwise.}\,\\
    
      \end{cases}
    \end{equation*}
Let $H$ be the column space of $A_n$, it is a subspace of $\bigwedge^{n-1}\Bbb R^{2n}$ and $\dim\,H=C_{n}$, the Catalan number. Denote by $\Bbb PH$ the projectivization of $H$. 
\end{definition}
We will be using later the following notations. The standard basis vectors of $\Bbb R^{2n}$ we will denote by $e_i$. For $I=(i_1, i_2, \dots, i_{n-1})$ denote by $e_{I}\in \bigwedge^{n-1}\Bbb R^{2n}$ a standard vector $e_{i_1}\wedge e_{i_2}\wedge \dots \wedge e_{i_{n-1}}$ and by $v_{\sigma}=\sum_{I}a_{I\sigma}e_{I}$  a column of $A_n$ labeled by $\sigma$.

Notice that  Definition \ref{def-elb} and  Definition \ref{def-grov2}  naturally generalized for the cactus networks, therefore we are able to refine Theorem \ref{elcon}:
\begin{theorem} \cite{L} \label{elconcom}
Let $e \in \overline{E}_n$ and $N(\Gamma, \omega)$ be the bipartite network associated with $e$   then the following hold: 
$$\Delta_{I}^M= \sum\limits_{\sigma} a_{\sigma I} L_{\sigma},$$
here summation is over all non-crossing partitions $\sigma.$ 
\end{theorem}

The proof of the next statement is an easy corollary of Theorem \ref{elconcom}.
\begin{theorem}  \cite{L} \label{Lthm}
The image of the set  $\overline E_n$ in 
$\mathrm{Gr}(n-1, 2n)$ belongs to the 
intersection $\mathrm{Gr}(n-1,2n)\cap\Bbb PH$. 
\end{theorem}

We will show that $H$ is invariant under the action of the group $\mathrm{Sp}(2n-2)$, moreover it can be identified with the fundamental representation of the group $\mathrm{Sp}(2n-2)$ which corresponds to the last vertex of the Dynkin diagram  for the Lie algebra $\mathfrak{sp}_{2n-2}$.

Recall some basic facts about this representation. Fix a symplectic form $\omega\in \bigwedge^2\Bbb R^{2k}$ and denote by  $\mathrm{Sp}(2k)$ the subgroup of elements of $SL(2k)$ which preserve $\omega$. We refer to it as the symplectic group. Representation theory of $\mathrm{Sp}(2k)$ is well known. The fundamental representations are irreducible representations labeled by the vertices of the Dynkin diagram for the Lie algebra $\mathfrak{sp}_{2n-2}$. The representation corresponding the last vertex is a sub-representation in the space of the fundamental representation $\bigwedge^k \Bbb C^{2k}$ of $SL(2k)$ and can be described as the kernel of the operator
\begin{equation}\label{def:Q}
Q:\bigwedge^k \Bbb C^{2k}\rightarrow \bigwedge^{k-2} \Bbb C^{2k}    
\end{equation}
given by the formula
$$Q(x_1\wedge x_2\cdots \wedge x_k)=\sum_{i<j} (-1)^{i+j}\omega(x_i,x_j)x_1 \wedge\cdots \wedge \hat x_i \wedge\cdots \wedge \hat x_j \wedge \cdots \wedge x_k$$
where $\omega$ is a symplectic form, see \cite{FH} for example.  Computing the dimension of the kernel we obtain
$$\binom{2k}{k}-\binom{2k}{k-2}=C_{k+1}$$

The appearance of the Lagrangian Grassmannian is very natural in this context.
\begin{lemma} $\mathrm{LG}(k,2k) = \mathrm{Gr}(k,2k)\cap P(\text{ker}\, Q)$.
\end{lemma}
\begin{proof}
Clearly $\mathrm{LG}(k,2k) \subset \mathrm{Gr}(k,2k) \cap P(\ker Q)$.

For the other inclusion, if $w \in \mathrm{Gr}(k,2k) \cap P(\ker Q)$, then $w$ is the equivalence class of $x_1 \wedge\ldots \wedge x_k$ with $x_1,\ldots,x_k$ linearly independent and $Q(w) = 0$, that is 
$$Q(w) = \sum_{i<j} (-1)^{i+j}\omega (x_i,x_j) x_1 \wedge\ldots \hat x_i \wedge\ldots \wedge \hat x_j \wedge\ldots\wedge x_k = 0$$ 
$1\leq i<j\leq k$ and thus $\omega (x_i,x_j) =0$ for all $1\leq i < j \leq k$.
\end{proof}

The above lemma suggests a way to refine the statement of the Theorem \ref{Lthm}. We have already proved  that the image of the set $E_n$ lands
in  
$\mathrm{Gr}(n-1,2n-2)\cong\mathrm{Gr}(n-1,V)\subset \mathrm{Gr}(n-1,2n)$. Now one would expect that the subspace $H$ is the kernel of the operator $Q$ and therefore is invariant under the action of the group $\mathrm{Sp}(2n-2)$. This is indeed the case.

\begin{theorem}\label{mainth}
The following holds
\begin{itemize}
    \item $H $ is the subspace of $\bigwedge^{n-1} V$. 
    \item  $H$ coincides with the kernel of the operator $Q$. 
\item The subspace $H$ is invariant under the action of the Lam group and hence the action of $\mathrm{Sp}(2n-2)$. Moreover $H$ is the the space of the fundamental representation of the group $\mathrm{Sp}(2n-2)$ which corresponds to the last vertex of the Dynkin diagram.
\end{itemize}
\end{theorem}
\begin{proof}
Note that the operator $Q$ defined above descents to the projectivization of the source and the target. Therefore it makes sense to apply it to points of $\mathrm{Gr}(n-1,2n)$.
Since the basis vectors of $H$ are labeled by $\sigma\in\mathcal{NC}_n$
we can prove the first two statements by induction on a number $r(\sigma)$ defined for $\sigma$ by the following formula:
$$r(\sigma)=n-c(\sigma),$$
where $c(\sigma)$ is the number of connected components of $\sigma$.

The basis of induction is provided by the partition $\sigma=(1|2|3|\dots|n)$ since it is the only partition with $r(\sigma)=0$.
Let  $e_{\emptyset} \in E_n$ be the empty electrical network then by direct calculation we conclude that 
\bea
X(e_{\emptyset})=(e_2+e_4)\wedge (e_4+e_6)\wedge \dots \wedge(e_{2n-2}+e_{2n})= w_2 \wedge w_4 \dots \wedge w_{2n-2}.
\eea
It is clear that $w_i\in V$ therefore $X(e_{\emptyset})\in\bigwedge^{n-1} V$  and by Theorem \ref{laggr} it belongs to the kernel of $Q$ as well. 


On the other hand $X(e_{\emptyset})=v_{1|2|3|\dots|n}$ so we conclude that the basis vector $v_{1|2|3|\dots|n}\in H$  belongs to $\bigwedge^{n-1}V$ and to the kernel of $Q$ providing therefore the basis of the induction.

Now let $\sigma$ be an arbitrary non-crossing partition and $v_\sigma\in H$ the corresponding basis vector. Consider any electrical network $e_{\sigma}(\Gamma, \omega) \in E_n$ such that the connected components of the graph $\Gamma$ induce on the boundary vertices a non-crossing partition equal to $\sigma$. The point $X(e_{\sigma})$ belong to $\mathrm{LG}(n-1)$ by Theorem \ref{laggr} therefore
\bea
Q(X(e_{\sigma}))=Q(\sum_{I} \Delta_I e_{I})=0.
\eea
Using Theorem \ref{elcon} rewrite the last identity as follows
\begin{equation}
  Q(\sum_{I} \Delta_I e_{I})=Q(\sum_{I}(\sum_{\delta\in\mathcal{NC}_n}a_{I\delta}L_{\delta})e_{I})=Q(L_{\sigma}v_{\sigma})+\sum_{\sigma'\not=\sigma}Q(L_{\sigma'}v_{\sigma'})=0.  
\end{equation}
It is easy to see that each $\sigma'$ with $L_{\sigma'}$ not equal to zero has the number $r(\sigma')$ strictly less than $\sigma$ so by the induction hypothesis each $v_{\sigma'}$ belongs to $\bigwedge^{n-1}V$ and to the kernel of $Q$, thus we obtain that $v_{\sigma}$ belongs to $\bigwedge^{n-1}V$ and to the kernel of $Q$ as well.

Since the subspace $H$ belongs to the kernel of the operator $Q$ and has the same dimension as the kernel it coincides with it. 

The restriction of  the Lam group action on the  subspace $V$ coincides by Theorem \ref{res} with the standard action of 
$\mathrm{Sp}(2n-2)$ on $\Bbb R^{2n-2}\cong V$. By the above the  kernel of the operator $Q$ and hence the subspace $H$ is the representation space of the fundamental representation of $\mathrm{Sp}(2n-2)$ which corresponds to the last vertex of the Dynkin diagram.  
\end{proof}

\begin{remark} Various explicit bases of this subrepresentation are of interest and have been studied a lot, see \cite{PLG}, \cite{DC} for example. 

It is interesting to compare the basis found in \cite{L} to the standard bases found in \cite{DC}.
\end{remark}
Putting it all together allows us to refine the statement of Theorem \ref{Lthm}.

\begin{theorem}
The image of the set $\overline{E}_n$ maps to the projectivisation of the compactification of the orbit of the highest weight vector $X(e_{\emptyset})$ in the fundamental representation of the group $\mathrm{Sp}(2n-2)$ which corresponds to the last vertex in the Dynkin diagram for the Lie algebra $\mathfrak{sp}_{2n-2}$. 
\end{theorem}

\section{Non-negative Lagrangian Grassmannian}\label{Sec:nonneg}

The choice of a symplectic form on $\Bbb R^{2n}$ leads to a specific embedding $\mathrm{Sp}(2n)\rightarrow SL(2n)$ and hence to a specific embedding $\mathrm {LG}(n)\rightarrow \mathrm {Gr}(n,2n)$.
Lusztig introduced the notion of the positive part in a homogeneous spaces, in particular, in the type $A$ flag variety and the Grassmannian $\mathrm {Gr}(n,2n)$.
This was described explicitly in \cite{MR}. 
According to \cite{K} there is a choice of the symplectic form such that the Lusztig non-negative part of $\mathrm {LG}(n)$ set-theoretically obeys the identity
$$\mathrm{LG}_{\geq 0}(n)=\mathrm{Gr}_{\geq 0}(n,2n)\cap \mathrm{LG}(n).$$

In this section we will take the above formula as the definition of the non-negative part of $\mathrm{LG}(n)$ for our choice of the symplectic form, which is different from the choice made in \cite{K}, and show that the points of $\mathrm{LG}(n-1)$ defined by $\widetilde{\Omega}_n(e)$ from (\ref{expansion}) are non-negative. We will prove it using the technique of adding boundary spikes and bridges developed in \cite{L}. 

Recall that according to our convention the operators $x_i(t), y_i(t)$ and $u_i(t)$ act on the right.
\begin{theorem} \cite{LT} \label{b-b}
Let $N(\Gamma, \omega)$ be a bipartite network with $n$ boundary vertices. Assume that $N'(\Gamma', \omega')$ is obtained from $N(\Gamma, \omega)$ by adding a bridge with the weight $t$, with a white vertex at $i$ and a black vertex at $i+1$. Then the points $X(N'),$ $X(N)
$ of $\mathrm{Gr}(n-1,2n)$
associated with $N'(\Gamma', \omega')$ and $N(\Gamma, \omega)$ accordingly are related to each other as follows: 
$$X(N')=X(N) x_i(t)$$ 

If $N''$ is obtained from $N(\Gamma, \omega)$ by adding a bridge  with the weight $t$, a black vertex at $i$ and a white vertex at $i+1$. Then the points $X(N'')$ and $X(N)$ of $\mathrm{Gr}(n-1,2n)$ associated with $N''(\Gamma'', \omega'')$ and $N(\Gamma, \omega)$ are related to each other as follows:   $$X(N'')=X(N) y_i(t),$$
where the matrices $x_i(t)$ and $y_i(t)$ introduced in Definition \ref{def:xy}.  

Thus adding any bridge with a positive weight preserves non-negativity.
\end{theorem}

\begin{remark} \label{s-change}
Let $e\in E_n$ and $e_{cl}\in E_n$ is obtained from $e$ by the clockwise shift by one of the indices of the boundary vertices then
$$X(N_{cl})=X(N)s_n$$
where the matrix $s_n$ was defined in \eqref{transpos}.
\end{remark}

\begin{definition} \cite{L}
For $e(\Gamma, \omega) \in E_n$, an arbitrary integer
$k\in \{1, 2,\dots , n\}$, and $t$ a non-negative real number define $u_{2k-1}(t)(e)$ to be the electrical network obtained from $e$ by adding a new edge with the weight $\frac{1}{t}$ connecting a new boundary vertex $v$ to the old boundary vertex labeled by $k$ which becomes an interior vertex. We call this operation \textit{adding a boundary spike}. 
\begin{figure}[h]
     \centering
     \includegraphics[scale=0.8]{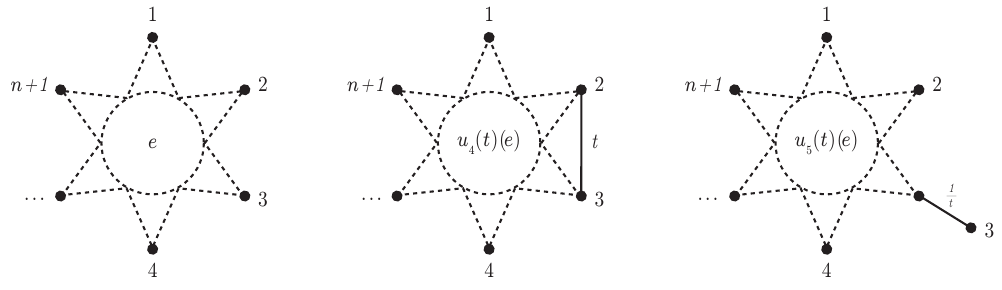}
     \caption{Adding a boundary bridge and a boundary spike}
     \end{figure}

For $e(\Gamma, \omega) \in E_n$, an arbitrary integer 
$k\in \{1, 2,\dots , n\}$, and $t$ a non-negative real number define $u_{2k}(t)(e)$ be the electrical network obtained from $e$ by adding a new edge from $k$ to $k + 1$ (indices taken modulo $n + 1$) with the weight $t$. We call this operation \textit{adding a boundary bridge}.
\end{definition}

\begin{figure}[h]
     \centering
     \includegraphics[scale=0.8]{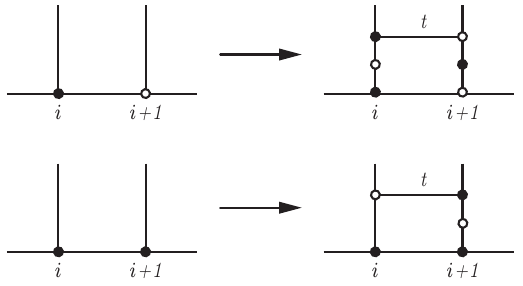}
     \caption{Adding a bridge}
     \end{figure}

\begin{theorem} \cite{L} \label{s-b}
Let $e(\Gamma, \omega) \in E_n$ be an electrical network  and the associated bipartite network $N_1(\Gamma, \omega)$. Likewise for the electrical networks $u_{2k-1}(t)(e)$ and $u_{2k}(t)(e) $ let   $N_2(u_{2k-1} \Gamma, u_{2k-1} \omega)$ and $N_2(u_{2k}  \Gamma, u_{2k} \omega)$ be the associated bipartite networks. Then, the points $X(N_1)$ and $X(N_2)$  of Grassmannian defined by  $N_1(\Gamma, \omega)$,  $N_2(u_{2k-1} \Gamma, u_{2k-1} \omega)$ and $N_2(u_{2k} \Gamma, u_{2k} \omega)$ are related to each other as follows:
$$X(N_2)=X(N_1)u_{2k-1}(t),  X(N_2)=X(N_1)u_{2k}(t),$$
where  $u_i(t)=x_i(t)y_{i+1}(t).$
\end{theorem}
We will illustrate this statement with a few particular examples to help the reader to see what is going on.

Consider the following $e\in E_n$ and the associated network $N(\Gamma, \omega).$ 

\begin{figure}[h]
     \centering
     \includegraphics[scale=0.7]{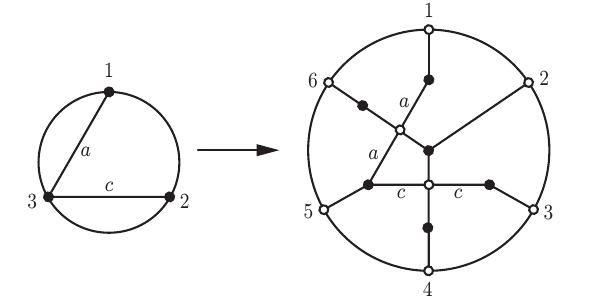}
     \caption{Electrical and its bipartite networks}
\end{figure}

Adding a boundary bridge between the vertices $1$ and $2$ we obtain the electrical network $u_2(b)(e)$ and associated bipartite network $N(u_2\Gamma, u_2\omega).$ 
\begin{figure}[h]
     \centering
     \includegraphics[scale=0.7]{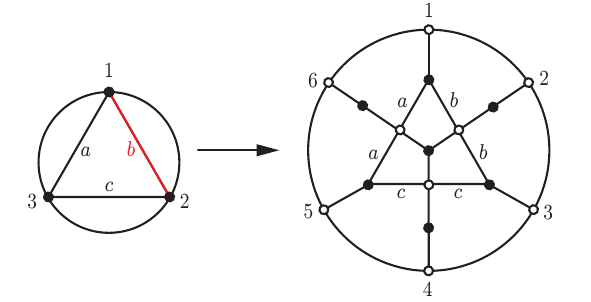}
     \caption{$u_2(b)(e)$ on the left, $N(u_2\Gamma, u_2\omega)$ on the right}
\end{figure}

Using Postnikov transformations \cite{LT} for bipartite networks we obtain that $N(u_2\Gamma, u_2\omega)$ is equivalent to the following bipartite network which can be obtained from the network $N(\Gamma,\omega)$ by adding two bridges with weights $b$. 
\begin{figure}[h]
     \centering
     \includegraphics[scale=0.7]{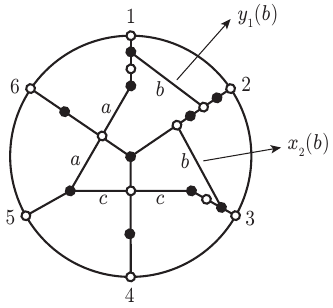}
     \caption{Network $N(\Gamma,\omega)$ with two added bridges}
\end{figure}
\newpage
Similarly for $e\in E_n$ on the left of Fig. \ref{fig:04} the network $u_3(\frac{1}{a})(e)$, on the left of Fig. \ref{fig:05}, is obtained by adding a spike to the vertex $3$ 

\begin{figure}[h]
     \centering
     \includegraphics[scale=0.9]{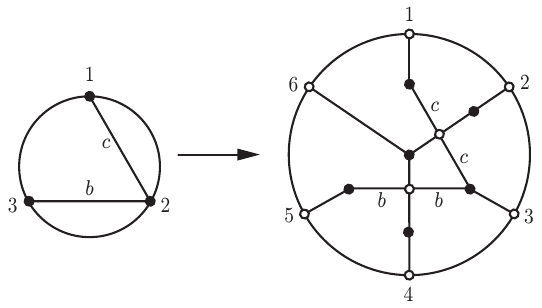}
     \caption{$e$ and associated bipartite network}
     \label{fig:04}
\end{figure}

\begin{figure}[h]
     \centering
     \includegraphics[scale=0.9]{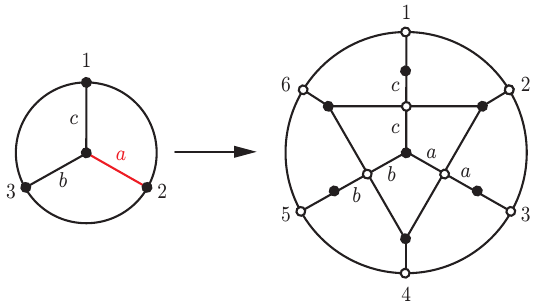}
     \caption{$u_3(\frac{1}{a})(e)$ and associated bipartite network $N(u_3\Gamma, u_3\omega)$}
     \label{fig:05}
\end{figure}

Again using the Postnikov transformations we obtain that the bipartite network  associated with $u_3(\frac{1}{a})(e)$ is equivalent to the following network which can be obtained from the network $N(\Gamma, \omega)$ by adding two bridges with weights $\frac{1}{a}$.
\begin{figure}[h]
     \centering
     \includegraphics[scale=0.9]{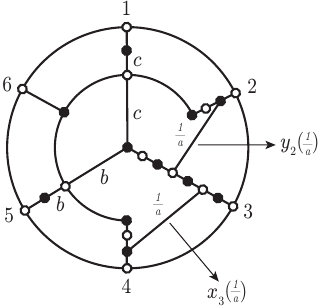}
     \caption{Equivalent to  $N(u_3\Gamma, u_3\omega)$ bipartite network}
     \label{fig:bipartite}
\end{figure}

It was proved in \cite{L} that the image of $\overline{E}_n$ lands in $\mathrm{Gr}_{\geq 0}(n-1,2n)$. Now we are ready to prove the following theorem:
\begin{theorem} \label{nonneglagr}
Let $e(\Gamma, \omega)\in E_n $ and $\widetilde{ \Omega}_n(e)$ is the point of $\mathrm{LG}(n-1)$ associated with it then $\widetilde{\Omega}_n(e)$  belongs to $\mathrm{LG}_{\geq 0}(n-1)$.
\end{theorem}

\begin{proof}
As it was demonstrated in Chapter $2.4 $ \cite{CIM1} each electrical network is equivalent to a critical electrical network \cite{Kl}. Moreover two equivalent electrical networks define the same point of  $\mathrm{Gr}_{\geq 0}(n-1, 2n)$ \cite{L} and as a consequence the same point of $\mathrm{LG}(n-1).$   Therefore it is sufficient to prove our statement only for the critical electrical  networks.

The proof is going by induction on the number of edges of $e \in E_n.$ Our argument is very similar to  Lam's argument for the  proof of the following statement from \cite{L}: for each point $X \in \mathrm{Gr}_{\geq 0}(k,n)$ there is a  bipartite network $N(\Gamma, \omega)$  which defines $X.$
Lam's original proof based on the fact that each point $X \in \mathrm{Gr}_{\geq 0}(k,n)$ can be "reduced".

Note that if $e$ has more than one connected component, the matrix  $\widetilde{\Omega}_n(e)$ has a block structure and it is sufficient to check non-negativity for each block, i.e. we can deal with each connected component of $e$ separately, therefore let us suppose that $e$ is connected.

The basis of induction is provided by the fact that each electrical network with three or less edges defines a positive point in $\mathrm{LG}(n-1).$ In fact we will prove a more general statement: each $e\in E_3$ defines  a point in $\mathrm{LG}_{\geq 0}(n-1).$ Consider an electrical network $e\in E_3$ and the associated point $\widetilde{\Omega}_3(e)\in $ which has the following form (see Example \ref{exlagr}): 
\begin{equation*}
\widetilde{\Omega}_3(e) = \left(
\begin{array}{ccccc}
-x_{21}& 1  & -x_{23}  & 0 & \\
x_{31}& 0  & -x_{31}-x_{32}  & 1 &  \\
\end{array}
\right).
\end{equation*}
The Plucker coordinates of $\widetilde{\Omega}_3(e)$ are equal to
 
  \begin{equation*}
\Delta_{12}(\widetilde{\Omega}_3(e))=-x_{31},\, \Delta_{13}(\widetilde{\Omega}_3(e))=x_{21}x_{31}+x_{21}x_{32}+x_{23}x_{31},\, \Delta_{14}(\widetilde{\Omega}_3(e))=-x_{21},
\end{equation*}
\begin{equation*}
\Delta_{23}(\widetilde{\Omega}_3(e))=-x_{31}-x_{32},\, \Delta_{24}(\widetilde{\Omega}_3(e))=1, \, \Delta_{34}(\widetilde{\Omega}_3(e))=-x_{23}.
\end{equation*}
Using the properties of a response matrix (see Theorem \ref{respc}) we obtain that each Plucker coordinate is non-negative.

Now we turn to the induction step. An edge of a critical electrical network $e\in E_n$ is called a boundary edge (respectively a boundary spike) if after its deletion  (respectively  contraction) we obtain an electrical network $e'$ such that $e=u_i(t)e,$ here $u_i(t)$ is the edge which we deleted (respectively  contracted).   Notice that $e'$ is again critical \cite{CIM1}.

Consider a connected critical electrical network $e\in E_n$ with more than three edges and more than three boundary vertices and the associated point $\widetilde{\Omega}_n(e).$ As it  was demonstrated  in \cite{CIM1} $e$ has at least three edges such that each of them is a boundary edge or a boundary spikes (called them boundary edges). Therefore there is a critical electrical network $e' \in E_n$ such that $e=e'u_i(t)$ here $i=1,\dots, 2n.$ Moreover, as we have at least three boundary edges we can choose one of them such that $i=2,\dots, 2n-1.$ 

By the induction assumption $\widetilde{\Omega}_n(e')$ defines a point in $\mathrm{LG}_{\geq 0}(n-1).$ Using the Theorem \ref{s-b} we conclude that 
$$\widetilde{\Omega}_n(e)=\widetilde{\Omega}_n(e')u_i|_V(t)$$
Now  notice that  each matrix $u_{i}(t)|_V$ $i=2, \dots, 2n-1$ preserves non-negativity as it follows from the Theorems \ref{b-b} and \eqref{expl-rest}. Hence we obtain that $\widetilde{\Omega}_n(e)$ defines a point in $\mathrm{LG}_{\geq 0}(n-1).$
\end{proof}

As a  corollary of the previous theorem we obtain the following:
\begin{corollary}
 The right action of matrices  $u_{i_j}(t_j)|_V$ $j=1, \dots, 2n$  and $s_n|_V$ preserve non-negativity of the  points $\widetilde{\Omega}_n(e).$
\end{corollary}
\begin{proof}
Consider two electrical networks $e' \in E_n$ and $e \in E_n$  such that $e'=eu_i(t)$ or $e'=es_n$ (see Remark \ref{s-change} for the geometrical meaning of the action of $s_n$). Using the Theorem \ref{s-b} we obtain that 
$$\widetilde{\Omega}_n(e')=\widetilde{\Omega}_n(e)u_i|_V(t)$$
or 
$$\widetilde{\Omega}_n(e')=\widetilde{\Omega}_n(e)s_n|_V(t).$$
Since we have already proved that $\widetilde{\Omega}_n(e')$ and $\widetilde{\Omega}_n(e)$ define points in $\mathrm{LG}_{\geq 0}(n-1)$  we conclude that matrices $u_{i}(t)|_V$ $i=1, \dots, 2n$ and $s_n|_V$ preserve non-negativity.   
\end{proof}

\begin{remark}
 Using the continuity arguments described in Section \ref{compactification} we conclude  that $X(e)\in \mathrm{LG}_{\geq 0}(n-1)$ for any $e\in \overline{E}_n$.
\end{remark}

\section{Applications of the new parametrisation}
\label{Sec:app}
We will give two applications of the parametrisation of the set $E_n$ as an orbit of the symplectic group action. Namely we will prove the well-known  formulas from \cite{CIM} for the change of the response matrix entries after adjoining a spike or an edge between boundary vertices and give a simple proof of Theorem 1.1 from \cite{KW} .
\subsection{Adjoining edges and the change of the response matrix}

\begin{theorem} \cite{CIM}, \cite{LP}
Let $e \in E_n $ be an electrical network. If $e' \in E_n $ is obtained from $e$ by adding an edge with the weight $t$  between the boundary vertices with the indices $k$ and $k+1$ then the following formula for the response matrix entries $x'_{ij}$ of $e'$ holds:
\begin{equation} \label{brf}
    x'_{ij}=x_{ij}+(\delta_{ik}-\delta_{(i+1)k})(\delta_{jk}-\delta_{(j+1)k})t,
\end{equation}

If $e' \in E_n $ is obtained from $e$ by adding a spike with the weight $\frac{1}{t}$  to a vertex with the index $k$ then the following formula for the entries $x'_{ij}$ of the response matrix of the network $e'$ holds:
\begin{equation} \label{sf}
    x'_{ij}=x_{ij}-\frac{tx_{ik}x_{kj}}{tx_{kk}+1},
\end{equation}
here $x_{ij}$ are the entries of the response matrix of $e$ and $\delta_{ij}$ are the Kronecker delta functions.
\end{theorem}
\begin{proof}
It is sufficient to prove \eqref{brf} in the case of adjoining an edge between the vertices with the indices  $1$ and $2$ because other cases can be obtained from it by the shift of enumeration of vertices.

By Theorem \ref{s-b} matrices $\Omega(e) u_2(t)$ and $\Omega(e')$ represent the same point of $\mathrm{Gr}(n-1,2n)$.

By a straightforward computation we get:
\begin{equation*} 
\Omega(e) u_2(t) = \left(
\begin{array}{cccccccccc}
x_{11}+t& 1  & -x_{12}+t  & 0 & x_{13}   &\ldots   & (-1)^n& \\
-x_{21}+t& 1  & t+x_{22}  & 1 & -x_{23}  &\ldots&     0&\\
x_{31}& 0 & -x_{32}  & 1 & x_{33} &\ldots & 0&\\
\vdots & \vdots & \vdots & \vdots  & \ddots &   \vdots & \vdots &   \\
(-1)^n x_{n1}& 0 & (-1)^{n+1} x_{n2}  & 0 & (-1)^{n+2} x_{n3} &\ldots&    1& \\
\end{array}
\right),
\end{equation*}
The multiplication by $u_2(t)$  does not change the columns of the matrix $\Omega(e)$ with even indices and therefore we obtain:
$$\Omega(e) u_2(t)=\Omega(e')$$ which implies \eqref{brf}.
The same argument works for each $u_{2k}(t)$.

For the proof of \eqref{sf} it is sufficient to obtain the formula in the case of  adjoining a spike to the  vertex of index  $2$.
As before from Theorem \ref{s-b} we conclude that the matrices $\Omega(e')$ and $\Omega(e) u_3(t)$ represent the same point of $\mathrm{Gr}(n-1,2n)$ and we get
\begin{equation*} 
Cal_2(t)\Omega(e) u_3(t) = \left(
\begin{array}{cccccccccc}
x_{11}-\frac{x_{12}x_{21}}{x_{22}t+1}& 1  & -x_{12}+\frac{x_{12}x_{22}}{x_{22}t+1} & 0 & x_{13}-\frac{x_{12}x_{23}}{x_{22}t+1}   &\ldots   & (-1)^n& \\
-x_{21}+\frac{x_{22}x_{21}}{x_{22}t+1}& 1  & x_{22}-\frac{x_{22}x_{22}}{x_{22}t+1}  & 1 & -x_{23}+\frac{x_{22}x_{32}}{x_{22}t+1}  &\ldots&     0&\\
x_{31}-\frac{x_{32}x_{12}}{x_{22}t+1}& 0 & -x_{32}+\frac{x_{32}x_{22}}{x_{22}t+1}  & 1 & x_{33}-\frac{x_{32}x_{23}}{x_{22}t+1} &\ldots & 0&\\
\vdots & \vdots & \vdots & \vdots  & \ddots &   \vdots & \vdots &   \\
\end{array}
\right),
\end{equation*}
where 
\begin{equation*} 
Cal_2(t)= \left(
\begin{array}{cccccccccc}
1& \frac{x_{12}}{1+x_{22}t}  & 0 & 0 &  \ldots   & 0 \\
0& \frac{1}{1+x_{22}t}  & 0 &  0 & \ldots   & 0 \\
0& \frac{x_{32}}{1+x_{22}t}  & 1 & 0 &  \ldots   & 0 \\
0& \frac{x_{42}}{1+x_{22}t}  & 0 &  1 & \ldots   & 0 \\
\vdots & \vdots & \vdots & \vdots  & \ddots &   \vdots &    \\
\end{array}
\right).
\end{equation*}
Therefore
$$Cal_2(t)\Omega(e) u_3(t)=\Omega(e').$$ which implies \eqref{sf}.
\end{proof}
\subsection{Kenyon-Wilson theorem}
We will show now that Theorem $1.1$ from  \cite{KW} could be obtained as a simple consequence of the main Theorem \ref{maint}.
\begin{theorem} \cite{KW} \label{kw}
Let $e=(\Gamma, \omega) \in E_n$ and $\sigma\in \mathcal{NC}$ 
then each fraction 
$$\frac{L_{\sigma}}{L}$$ 
can be expressed as a homogeneous polynomial  of degree $n-k$ in variables $L_{i;j}:=\frac{L_{ij}}{L}$ with integer coefficients where $k$ is a number of connected components of $\sigma$.   
\end{theorem}

For simplicity we consider only the case when $\sigma\in \mathcal{NC}$, but the statement is true even for an arbitrary $\sigma$ \cite{KW}.
The following is essentially the statement of Proposition $5.19$ from \cite{L}:
\begin{proposition} \label{gro}
Let $e=(\Gamma, \omega) \in E_n$ be an electrical network. We associated to it a bipartite network $N(\Gamma, \omega)$. Then for each  $\sigma\in \mathcal{NC}$ there is a set of coordinates $\Delta^d_{J(\sigma)}$ such that the following identity holds: 
$$L_{\sigma}=\sum \limits_{J(\sigma)} a_{J(\sigma)}\Delta^d_{J(\sigma)},$$
with $a_{J(\sigma)}=\{1, -1\}$.
Moreover for each $J(\sigma)$ the following holds: $|J(\sigma)_{odd}|=n-k$ where $|J(\sigma)_{odd}|$ is the number of odd indices in $J(\sigma).$ 
\end{proposition}

Now we are ready to prove Theorem \ref{kw}. Let  $e=(\Gamma, \omega) \in E_n$ be an electrical network and the associated bipartite network $N(\Gamma, \omega)$. Introduce a perfect orientation $O$ on $N(\Gamma, \omega)$ to obtain the Postnikov network $NP(\Gamma, \omega', O)$. For $\sigma\in\mathcal{NC}$ and the appropriate coordinate $L_{\sigma}$ according to Theorems  \ref{contal}  and \ref{gro} the following holds:    
$$\frac{L_{\sigma}}{L}=\sum \limits_{J(\sigma)} \frac{a_{J(\sigma)}\Delta^d_{J(\sigma)}}{L}=\sum \limits_{J(\sigma)}a_{J(\sigma)}\Delta_{J(\sigma)}(A),$$
where $A$ is the extended matrix of boundary measurements for $NP(\Gamma, \omega', O).$

On the other hand for each $\Delta_{J(\sigma)}$ using Theorem \ref{maint}  we get:  $$\frac{\Delta_{J(\sigma)}(A)}{\Delta_{K}(A)}=\frac{\Delta_{J(\sigma)}(\Omega'_n(e))}{\Delta_{K}(\Omega'_n(e))},$$
where $K=\{2,4,6, \dots, 2n-2\}$. Using that $\Delta_K(A)=\Delta_K(\Omega'_n)=1$ we obtain the following identities: 
$$\frac{L_{\sigma}}{L}=\sum \limits_{J(\sigma)}a_{J(\sigma)}\Delta_{J(\sigma)}(A)=\sum \limits_{J(\sigma)}a_{J(\sigma)}\frac{\Delta_{K}(A)\Delta_{J(\sigma)}(\Omega'_n(e))}{\Delta_{K}(\Omega'_n(e))}=\sum \limits_{J(\sigma)}a_{J(\sigma)}\Delta_{J(\sigma)}(\Omega'_n(e)).$$
Notice that the last expression is a homogeneous polynomial in the variables $x_{ij}$ with integer coefficients and therefore according to Theorem \ref{kenwil} is a polynomial in $L_{i;j}$, since $L_{i;j}=-x_{ij}.$  It is easy to see that the degree of each summand $\Delta_{J(\sigma)}(\Omega'_n(e))$ is equal to $|J(\sigma)_{odd}|$. The last number is equal to the number of columns with odd indices in the matrix $\Omega'_n(e)$. Recall that the columns with even indices contain only $0$ and $1$.  Now  using Theorem \ref{gro} we conclude that  $|J(\sigma)_{odd}|=n-k.$ 
\begin{example}
Consider an electrical network $e \in E_n$ and the associated bipartite network $N(\Gamma, \omega).$ Using Theorem \ref{elcon} it is easy to check that for $N(\Gamma, \omega)$ the following holds:  
$$\Delta_{567}^M= L_{14|23}.$$
Therefore by the proof of Theorem \ref{kw} we obtain that 
$$\frac{L_{14|23}}{L}=\Delta_{567}(\Omega'_4(e))=x_{14}x_{23}-x_{24}x_{13}=L_{14}L_{23}-L_{24}L_{13}.$$
\end{example}

\section{Vertex representation and standard well-connected graphs }
\label{Sec:vertex}

\subsection{Vertex approach} The authors of \cite{GT} defined for $e\in E_n$ the boundary partition function $M_B(e)$ or simply $M_B$ if it is clear which $e$ we work with. It is a matrix which depends on at most $n(n-1)/2$ parameters. These parameters are the conductivities of the edges of the network.  Such a matrix according to \cite{GT} belongs to the symplectic group $\mathrm{Sp}(n)$ where $n$ can be odd or even.


\bea
T_{2n}=\left(
\begin{array}{ccccccc}
1 & 0 & 0 &\cdots & 0 & 0 &\cdots\\
0 & 0 & 0 & \cdots & 1 & 0 & \cdots \\
0 & 1 & 0 & \cdots & 0 & 0 & \cdots \\
0 & 0 & 0 & \cdots & 0 & 1 & \cdots \\
0 & 0 & 1 & \cdots & 0 & 0 & \cdots \\
\vdots &  \vdots & \vdots & \ddots & \vdots & \vdots & \vdots\\
0 & 0 & 0& \cdots & 0 & 0 & \ddots
\end{array}
\right);\qquad
S_n=\left(
\begin{array}{ccccc}
1 & 0 &0 &  \cdots & -1\\
-1 & 1 & 0 &  \cdots & 0 \\
0 &-1 & 1 &  \cdots & 0 \\
\vdots &  \vdots & \vdots & \ddots & \vdots\\
0 & 0 & 0 & \cdots & 1
\end{array}
\right).\nn
\eea
One of the main results from \cite{GT} is
\begin{lemma} \label{lemma:w1w2}
For an electrical network $e \in E_n$ on the standard well-connected graph $\Sigma_n$ the row spaces of the matrices
\bea
W_1=(S_n,M_R)\quad \mbox{and} \quad
W_2=(M_B,Id_n)S_{2n} T_{2n}\nn
\eea
define the same point in $\mathrm{Gr}(n-1,2n)$.
\end{lemma}

We define standard well-connected graphs by induction, the first few examples are presented on Figure \ref{stgr}. 

\begin{figure}[h]
\centering
\includegraphics[width=80mm]{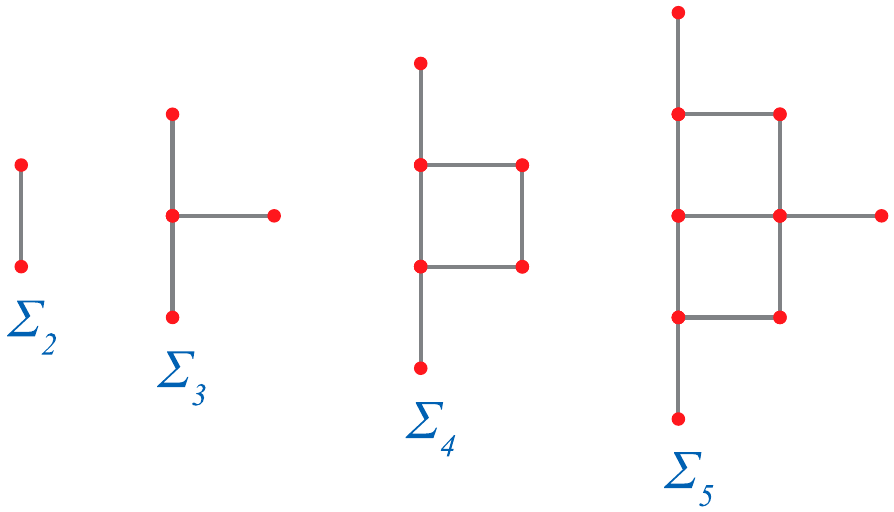}
\caption{Standard well-connected graphs}
\label{stgr}
\end{figure}

Lemmas \ref{lemma:w1w2}, \ref{vertex model} and \ref{vertex model 2} imply immediately:

\begin{corollary}\label{boundary}
The subspace $W_2$ defines a point in $\mathrm{LG}(n-1,V')$ for any value of the parameters.

Moreover the presentation of  electrical networks on standard well-connected graphs given in  \cite{GT} and  the presentation given in \cite{L}  restricted on electrical networks on standard well-connected graphs  inside of the projective space $\mathbb P(\bigwedge^{n-1}\mathbb R^{2n})$ differ by the automorphism of the ambient space $\mathbb R^{2n}$ given by the matrix $\overline{T}_{2n}$ from Lemma \ref{vertex model 2}.
\end{corollary}
\begin{proof}
The first statement follows from the fact that the subspace $\Omega^{aux}_e$ and $W_1$ coincide.

The second statement is a direct consequence of the formula given in Lemma
\ref{vertex model 2}
\end{proof}

\subsection{Standard well-connected graph technique} 
In this subsection we sketch another strategy for proving Theorem \ref{laggr} and Theorem \ref{nonneglagr}.

First of all let us prove Theorem \ref{laggr} and Theorem \ref{nonneglagr}   only for electrical networks $e(\Sigma_n, \omega)$ on the  standard well-connected graphs (see Figure \ref{stgr}). Consider the electrical network  $e(\Sigma_n, \omega).$ Notice that it can be obtained from the empty network $e_{\emptyset} \in E_n$ by the sequence $\{u_{i_j}(t_j)\}$ $j=2, \dots, 2n-1$ of adding spikes and bridges (for example, the electrical network $e(\Sigma_4, \omega)$ can be obtained by these  operations as it is shown in  Figure \ref{fig:adding}, each electrical networks $e(\Sigma_n, \omega)$ can be obtained by the same way). Using Theorem \ref{s-b} we conclude that 

$$\widetilde{\Omega}_n(e)=\widetilde{\Omega}_n(e_{\emptyset})\prod_{j=1}^m(u_{i_j}(t_j)|_V),$$
here $\widetilde{\Omega}_n(e_{\emptyset})$ is the point of $\mathrm{LG}(n-1)$ associated with the empty network $e_{\emptyset}.$ It is easy to see that the point $\widetilde{\Omega}_n(e_{\emptyset})$ belongs to the non-negative part of $\mathrm{LG}(n-1,V').$ Since each $u_{i_j}(t_j)|_V$ $j=2, \dots, 2n-1$  is symplectic (see Theorem \ref{res}) we obtain that $\widetilde{\Omega}_n(e)$ belongs to $\mathrm{LG}(n-1,V').$   Also as we already obtained (see the proof of  Theorem \ref{nonneglagr})  each  $u_{i_j}(t_j)|_V$ $j=2, \dots, 2n-1$  preserves non-negativity. Therefore we conclude that
$$\widetilde{\Omega}_n(e)=\widetilde{\Omega}_n(e_{\emptyset})\prod_{j=1}^m(u_{i_j}(t_j)|_V)$$
is the non-negative point of $\mathrm{LG}(n-1,V').$ 

\begin{figure}[h]
     \centering
     \includegraphics[scale=0.4]{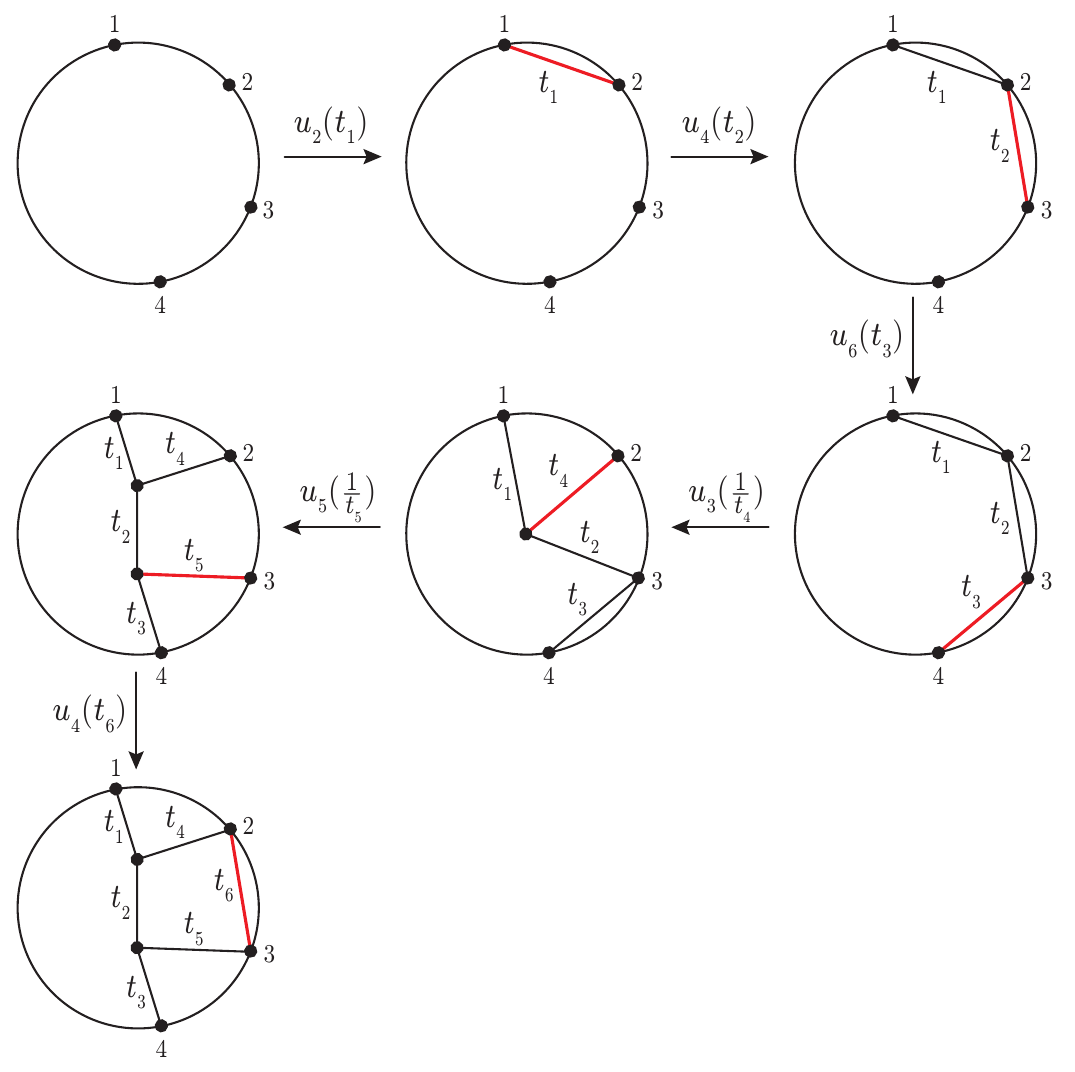}
     \caption{Adding boundary spikes and  boundary bridges for $\Sigma_4$}
     \label{fig:adding}
\end{figure}

Now it remains to extend the above statement for all networks. It can be done  by using the fact from \cite{Kl} that any network is electrical equivalent to a network on a minor of the standard well-connected graph (see Figure \ref{stgr}), where minor is a subgraph of a graph after deletion or contraction some set of edges. As we explained in Section \ref{compactification} edge deletion and contraction corresponds to passing to the limit when some of the edge conductivities are approaching zero or infinity, therefore using the continuity argument we conclude that all networks define points in $\mathrm{LG}(n-1,V')$ (we get Theorem \ref{laggr}). Moreover all these points are  non-negative (we get Theorem \ref{nonneglagr}).

So in fact it is sufficient to prove some theorems only for networks on standard well-connected  graphs. Moreover this idea might be useful for working with other classes of dimer models, in particular for studying the Ising model. We are planning to develop this approach in a future publication.

\section*{Acknowledgements} 
We are grateful to Arkady Berenstein and Alexander Kuznetsov for useful discussions and to Lazar Guterman for pointing out few misprints. Research of B.B, A.K. and D.T. (Sections 2, 3 and 4) has been supported in part by the RSF Grant No. 20-71-10110 which funds the work of B.B, A.K. and D.T. at P. G. Demidov Yaroslavl State University. Research of B.B. (Sections 7 and 8) was supported in part by the ISF grant 876/20.
Part of the work in this project was completed while the second author visited the ICERM, USA and the MPI, Bonn, Germany.

\end{document}